\documentclass[11pt,times,twoside]{article}

\usepackage{graphicx,latexsym,euscript,makeidx,color}
\usepackage{amsmath,amsfonts,amssymb,amsthm,mathrsfs}
\usepackage{enumerate}

\usepackage[colorlinks,linkcolor=blue,anchorcolor=green,citecolor=red]{hyperref}%when print, use the next package
\usepackage{geometry}
\def\5n{\negthinspace \negthinspace \negthinspace \negthinspace \negthinspace }
\def\4n{\negthinspace \negthinspace \negthinspace \negthinspace }
\def\3n{\negthinspace \negthinspace \negthinspace }
\def\2n{\negthinspace \negthinspace }
\def\1n{\negthinspace }
 \def\sC{\mathscr{C}}    
     
\def\dbE{\mathbb{E}}     
\def\dbF{\mathbb{F}}   \def\cF{{\cal F}}  
     
\def\dbH{\mathbb{H}}

\def\dbP{\mathbb{P}}     
     
\def\dbR{\mathbb{R}}   \def\cR{{\cal R}}  
\def\dbS{\mathbb{S}}   \def\cS{{\cal S}}  
     
   \def\cU{{\cal U}}

\def\Om{\Omega}           \def\Th{\Theta}
\def\ss{\smallskip}                
\def\ms{\medskip}         \def\da{\mathop{\downarrow}}       
         \def\Ra{\mathop{\Rightarrow}}      \def\lan{\langle}
\def\ds{\displaystyle}           \def\ran{\rangle}
\def\ra{\rightarrow}      
 
\def\no{\noindent}        \def\q{\quad}                      \def\llan{\left\langle}
\def\ns{\noalign{\ss}}    \def\qq{\qquad}                    \def\rran{\right\rangle}
    \def\hb{\hbox}                     
                   
         \def\rf{\eqref}                    
  \def\deq{\triangleq}               
\def\tb{\textcolor{blue}} \def\ae{\hbox{\rm a.e.}}           \def\({\Big (}
\def\les{\leqslant}                  \def\){\Big )}
\def\ges{\geqslant}       \def\esssup{\mathop{\rm esssup}}   \def\[{\Big[}
           \def\]{\Big]}
          \def\tr{\hbox{\rm tr$\,$}}         \def\cd{\cdot}
\def\wt{\widetilde}              \def\cds{\cdots}
                            
        \def\ts{\times}
\def\a{\alpha}              \def\O{\Omega}   
         \def\D{\Delta}   \def\d{\delta}        
\def\z{\zeta}         \def\Th{\Theta}      \def\si{\sigma}
\def\e{\varepsilon}     \def\l{\lambda}        
         \def\f{\varphi}  \def\i{\infty}   

\def\bde{\begin{definition}\label}    \def\ede{\end{definition}}
\def\be{\begin{equation}}
\def\bel{\begin{equation}\label}      \def\ee{\end{equation}}
\def\bt{\begin{theorem}\label}        \def\et{\end{theorem}}
\def\bc{\begin{corollary}\label}      \def\ec{\end{corollary}}
\def\bl{\begin{lemma}\label}          \def\el{\end{lemma}}
\def\bp{\begin{proposition}\label}    \def\ep{\end{proposition}}
\def\bas{\begin{assumption}\label}    \def\eas{\end{assumption}}
\def\br{\begin{remark}\label}         \def\er{\end{remark}}
\def\bex{\begin{example}\label}       \def\ex{\end{example}}
\def\ba{\begin{array}}                \def\ea{\end{array}}
\def\ben{\begin{enumerate}}           \def\een{\end{enumerate}}

\newtheorem{theorem}{Theorem}[section]
\newtheorem{definition}[theorem]{Definition}
\newtheorem{proposition}[theorem]{Proposition}
\newtheorem{corollary}[theorem]{Corollary}
\newtheorem{lemma}[theorem]{Lemma}
\newtheorem{remark}[theorem]{Remark}
\newtheorem{example}[theorem]{Example}

\makeatletter
   
   \@addtoreset{equation}{section}
\makeatother

\allowdisplaybreaks[4]
\begin{document}

\title{\bf Weak Closed-Loop Solvability of Stochastic Linear Quadratic Optimal Control Problems of Markovian Regime Switching System\thanks{This work is supported partially by Southern University of Science and Technology Start up fund Y01286120 and National Natural Science Foundation of China grants 61873325 and 11831010. Also, this work is partially supported by Research Grants Council of Hong Kong under grants 15255416 and 15213218.}}

\author{Jiaqiang Wen\thanks{Department of Mathematics, Southern University of Science and Technology,
Shenzhen, Guangdong, 518055, China (wenjq@sustech.edu.cn).}~,~~~
Xun Li\thanks{Corresponding author. Department of Applied Mathematics, The Hong Kong Polytechnic University,
Hong Kong (malixun@polyu.edu.hk).}~,~~~
Jie Xiong\thanks{Department of Mathematics and SUSTech International center for Mathematics, Southern University of Science and Technology, Shenzhen, Guangdong, 518055, China (xiongj@sustech.edu.cn).}}

%\date{}
\maketitle

%--------------------------------------------------------------------------------------------------------

\no\bf Abstract. \rm
In this paper, we investigate open-loop and weak closed-loop solvabilities of stochastic linear quadratic (LQ, for short) optimal control problem of Markovian regime switching system. Interestingly, these two solvabilities are equivalent on $[0, T)$. We first provide an alternative characterization of the open-loop solvability of LQ problem using a perturbation approach. Then, we study the weak closed-loop solvability of LQ problem of Markovian regime switching system, and establish the equivalent relationship between open-loop and weak closed-loop solvabilities. Finally, we present an example to shed on light on finding weak closed-loop optimal strategies within the framework of Markovian regime switching system.

\ms

\no\bf Key words: \rm
stochastic linear quadratic optimal control, Markovian regime switching, Riccati equation, open-loop solvability, weak closed-loop solvability.
%feedback

\ms

\no\bf AMS subject classifications. \rm
93E20, 49N10, 49N35.

\section{Introduction}

Linear quadratic (LQ, for short) optimal control can be traced back to the works of Kalman \cite{Kalman 1960} for the deterministic cases, and Wonham \cite{Wonham 1968} for the stochastic cases (also see Davis \cite{Davis 1977}, Bensoussan \cite{Bensoussan 1982}, Yong and Zhou \cite{Yong-Zhou 1999} and the references therein). In the classical setting, under some mild conditions on the weighting coefficients in the cost functional such as positive definiteness of the weighting control matrix, the stochastic LQ optimal control problems can be solved elegantly via Riccati equation approach (see Yong and Zhou \cite[Chapter 6]{Yong-Zhou 1999}). Chen, Li and Zhou \cite{Chen-Li-Zhou 1998} studied stochastic LQ optimal control problems with indefinite weighting control matrix as well as financial applications such as continuous time mean-variance portfolio selection problems (see Zhou and Li \cite{Zhou-Li 2000}, Li, Zhou and Lim \cite{Li-Zhou-Lim 2002}). Since then, there has been an increasing interest in the so-called {\it indefinite} stochastic LQ optimal control (see Ait Rami, Moore and Zhou \cite{Ait Rami-Moore-Zhou 2001}, Li, Zhou and Ait Rami \cite{Li-Zhou-Ait Rami 2003}).

\ms

A topic of state systems involving random jumps, such as Poisson jumps or regime switching jumps, is of interest and of importance in various fields such as engineering, management, finance, economics, and so on. For example, Wu and Wang \cite{Wu-Wang 2003} considered the stochastic LQ optimal control problems with Poisson jumps and obtained the existence and uniqueness of the deterministic Riccati equation. Using the technique of completing squares, Hu and Oksendal \cite{Hu-Oksendal 2008} discussed the stochastic LQ optimal control problem with Poisson jumps and partial information. Yu \cite{Yu 2017} investigated a kind of infinite horizon backward stochastic LQ optimal control problems. Li, Wu and Yu \cite{Li-Wu-Yu 2018} solved the indefinite  stochastic LQ optimal control problem with Poisson jumps.
%Meanwhile, the stochastic control problems involving regime switching jumps have been extensively studied in the areas of applied probability and stochastic controls.
Meanwhile, there has been dramatically increasing interest in studying this family of stochastic control problems as well as their financial applications, see, for examples, Zhou and Yin \cite{Zhou-Yin 2003,Yin-Zhou 2004}, Zhang, Siu and Meng  \cite{Zhang-Siu-Meng 2010}, Zhang, Elliott and Siu \cite{Zhang-Elliott-Siu 2012}, Zhang, Sun and Xiong \cite{Zhang-Sun-Xiong 2018}, Mei and Yong \cite{Mei-Yong 2018}, Hu, Liang and Tang \cite{Hu-Liang-Tang 2018} and Sun and Wang \cite{Sun-Wang 2019}. Moreover, Ji and Chizeck \cite{Ji-Chizeck 1990,Ji-Chizeck 1992} formulated a class of continuous-time LQ optimal controls with Markovian jumps. Zhang and Yin \cite{Zhang-Yin 1999} developed hybrid controls of a class of LQ systems modulated by a finite-state Markov chain. Li, Zhou and Ait Rami \cite{Li-Zhou-Ait Rami 2003} initiated indefinite stochastic LQ optimal controls with regime switching jumps. Liu, Yin and Zhou \cite{Liu-Yin-Zhou 2005} considered near-optimal controls of regime switching LQ problems with indefinite control weight costs. Some other recent development concerning regime switching jumps see Donnelly and Heunis \cite{DH12}, and Li and Zheng \cite{LZ15}.

\ms

Recently, Sun and Yong \cite{Sun-Yong 2014} investigated the two-person zero-sum stochastic LQ differential games. It was shown in \cite{Sun-Yong 2014} that the open-loop solvability is equivalent to the existence of an adapted solution to a forward-backward stochastic differential equation (FBSDE, for short) with constraint and the closed loop solvability is equivalent to the existence of a regular solution to the Riccati equations. As a continuation work of \cite{Sun-Yong 2014}, Sun, Li and Yong \cite{Sun-Li-Yong 2016} fundamentally studied the open-loop and closed-loop solvabilities for stochastic LQ optimal control problems. Moreover, the equivalence between the strongly regular solvability of the Riccati equation and the uniform convexity of the cost functional is established. Wang, Sun and Yong \cite{Wang-Sun-Yong 2018} introduced the notion of weak closed-loop optimal strategy for LQ problems, and obtained its existence which is equivalent to the open-loop solvability of the LQ problem. Zhang, Li and Xiong \cite{Zhang-Li-Xiong 2018} studied the open-loop and closed-loop solvabilities for stochastic LQ optimal control problems with Markovian regime switching jumps, and established the equivalent relationship between the strongly regular solvability of the Riccati equation and the uniform convexity of the cost functional in the circumstance of Markovian regime switching system. In this paper, we further study the \textit{weak} closed-loop solvability of stochastic LQ optimal control problems with Markovian regime switching system. In order to present our work more clearly, we describe the problem in detail.

\ms

Let $(\O,\cF,\dbF,\dbP)$ be a complete filtered probability space on
which a standard one-dimensional Brownian motion $W=\{W(t); 0\les t
< \i \}$ and a continuous time, finite-state, Markov chain $\a=\{\a(t); 0\les t< \i \}$ are defined, where $\dbF=\{\cF_t\}_{t\ges0}$ is the natural filtration of $W$ and $\a$ augmented by all the $\dbP$-null sets in $\cF$,
and $\dbF^\a=\{\cF_t^\a\}_{t\ges0}$ is the filtration generated by $\a$, with the related expectation $\dbE^\a$.
We identify the state space of the chain $\a$ with a finite set $\cS\deq\{1, 2 \dots, D\}$, where $D\in \mathbb{N}$ and suppose that the chain is homogeneous and irreducible. Let $0\les t<T$ and consider the following controlled Markovian
regime switching linear stochastic differential equation (SDE, for short) over a finite time horizon $[t,T]$:
\bel{state}\left\{\ba{ll}
\ds dX(s)=\Big[A(s,\a(s))X(s)+B(s,\a(s))u(s)+b(s)\Big]ds\\
\ns\ns\ds\qq\qq +\Big[C(s,\a(s))X(s)+D(s,\a(s))u(s)+\si(s)\Big]dW(s),\q~s\in[t,T],\\
\ns\ns\ds X(t)=x,\q~\a(t)=i,
\ea\right.\ee
where $A,C:[0,T]\ts\cS\ra\dbR^{n\ts n}$ and $B,D:[0,T]\ts\cS\ra\dbR^{n\ts m}$ are given deterministic functions, called the coefficients of the {\it state equation} \rf{state};
$b,\si:[0,T]\ts\Om\ra\dbR^n$ are $\dbF$-progressively measurable processes,
called the {\it nonhomogeneous terms}; and $(t,x,i)\in[0,T)\ts\dbR^n\ts\cS$ is called the {\it initial pair}. In the above, the process $u(\cd)$, which belongs to the following space:
$$\cU[t,T]\deq\Big\{u:[t,T]\ts\Om\ra\dbR^m\ \Big|\ u\hb{ is } \dbF\hb{-progressively measurable and }
 \dbE\int_t^T|u(s)|^2ds<\i\Big\},$$
is called the {\it control process}, and the solution $X(\cd)$ of \rf{state} is called the {\it state process} corresponding to $(t,x,i)$ and $u(\cd)$. To measure the performance of the control $u(\cd)$, we introduce the following quadratic {\it cost functional}:
\bel{cost}\ba{ll}
\ds J(t,x,i;u(\cd))\deq\dbE\Bigg\{\Big\lan G(\a(T))X(T),X(T)\Big\ran+2\Big\lan g(\a(T)),X(T)\Big\ran\\
\ns\ns\ds\qq\qq\qq\qq \  \,+\int_t^T\Bigg[\llan\begin{pmatrix}Q(s,\a(s))&S(s,\a(s))^\top\\S(s,\a(s))&R(s,\a(s))\end{pmatrix}
                                \begin{pmatrix}X(s)\\ u(s)\end{pmatrix},
                                \begin{pmatrix}X(s)\\u(s)\end{pmatrix}\rran \\
\ns\ns\ds\qq\qq\qq\qq\qq\qq+2\llan\begin{pmatrix}q(s,\a(s))\\ \rho(s,\a(s))\end{pmatrix},\begin{pmatrix}X(s)\\ u(s)\end{pmatrix}\rran \Bigg]ds\Bigg\},
\ea\ee
where $G(i)\in\dbR^{n\ts n}$ is a symmetric constant matrix, and $g(i)$ is an $\cF_T$-measurable random variable
taking values in $\dbR^n$, with $i\in\cS$; $Q:[0,T]\ts\cS\ra\dbR^{n\ts n}$, $S:[0,T]\ts\cS\ra\dbR^{m\ts n}$ and $R:[0,T]\ts\cS\ra\dbR^{m\ts m}$ are deterministic functions with both $Q$ and $R$ being symmetric; $q:[0,T]\ts\cS\ra\dbR^{n}$ and $\rho:[0,T]\ts\cS\ra\dbR^{m}$ are deterministic functions. In the above, $M^\top$ stands for the transpose of a matrix $M$. The problem that we are going to study is the following:

\ms

\bf Problem (M-SLQ). \rm For any given initial pair $(t,x,i)\in[0,T)\ts\dbR^n\ts\cS$, find a control $u^{*}(\cd)\in\mathcal{U}[t,T]$, such that
\bel{optim}J(t,x,i;u^{*}(\cd))=\inf_{u(\cd)\in\cU[t,T]}J(t,x,i;u(\cd)),\q~\forall u(\cd)\in\mathcal{U}[t,T].\ee
The above is called a {\it stochastic linear quadratic optimal control problem} of the Markovian regime switching system. Any $u^{*}(\cd)\in\mathcal{U}[t,T]$ satisfying \rf{optim} is called an {\it open-loop optimal control} of Problem (M-SLQ) for the initial pair $(t,x,i)$; the corresponding state process $X(\cd)=X(\cd\ ;t,x,i,u^*(\cd))$ is called an {\it optimal state process}; and the function $V(\cd,\cd,\cd)$ defined by
\bel{value} V(t,x,i)\deq\inf_{u(\cd)\in\cU[t,T]}J(t,x,i;u(\cd)),\q~(t,x,i)\in[0,T]\ts\dbR^n\ts\cS,\ee
is called the {\it value function} of Problem (M-SLQ).
Note that in the special case when $b(\cd,\cd),\si(\cd,\cd),g(\cd),q(\cd,\cd),\rho(\cd,\cd)=0$, the state equation \rf{state} and the cost functional \rf{cost}, respectively, become
\bel{state-0}\left\{\ba{ll}
\ds dX(s)=\Big[A(s,\a(s))X(s)+B(s,\a(s))u(s)\Big]ds\\
\ns\ns\ds\qq\qq +\Big[C(s,\a(s))X(s)+D(s,\a(s))u(s)\Big]dW(s),\q~s\in[t,T],\\
\ns\ns\ds X(t)=x,\q~\a(t)=i,
\ea\right.\ee
and
\bel{cost-0}\ba{ll}
\ds J^0(t,x,i;u(\cd))=\dbE\Bigg\{\Big\lan G(\a(T))X(T),X(T))\Big\ran\\
\ns\ns\ds\qq\qq\qq\qq\q+\int_t^T\llan\begin{pmatrix}Q(s,\a(s))&S(s,\a(s))^\top\\S(s,\a(s))&R(s,\a(s))\end{pmatrix}
                                \begin{pmatrix}X(s)\\ u(s)\end{pmatrix},
                                \begin{pmatrix}X(s)\\u(s)\end{pmatrix}\rran ds\Bigg\}.
\ea\ee
We refer to the problem of minimizing \rf{cost-0} subject to \rf{state-0} as the homogeneous LQ problem associated with Problem (M-SLQ), denoted by \bf Problem (M-SLQ)$^0$\rm. The corresponding value function is denoted by $V^0(t,x,i)$.
Moreover, when all the coefficients of \rf{state} and \rf{cost} are independent of the regime switching term $\a(\cd)$, the corresponding problem \rf{optim} is called \bf Problem (SLQ)\rm.  %And similarly, \bf Problem (SLQ)$^0$\rm.

\ms

Following the works of \cite{Sun-Yong 2014,Sun-Li-Yong 2016}, Zhang, Li and Xiong \cite{Zhang-Li-Xiong 2018} investigated the open-loop and closed-loop solvabilities for stochastic LQ problems of Markovian regime switching system. It was shown that the open-loop solvability of Problem (M-SLQ) is equivalent to the solvability of a forward-backward stochastic differential equation with constraint. They also showed that the closed-loop solvability of Problem (M-SLQ) is equivalent to the existence of a {\it regular} solution of the following general Riccati equation (GRE, for short):
\bel{Riccati-0.1}\left\{\ba{ll}
\ds \dot P(s,i)+P(s,i)A(s,i)+A(s,i)^\top P(s,i)+C(s,i)^\top P(s,i)C(s,i)+Q(s,i)\\
\ns\ds\qq\q -\hat{S}(s,i)^\top\hat{R}(s,i)^{-1}\hat{S}(s,i)+\sum_{k=1}^D\l_{ik}(s)P(s,k)=0,\q \ae~s\in[0,T],\ i\in\cS,\\
\ns\ds P(T,i)=G(i),
\ea\right.\ee
where
\begin{equation*}
  \begin{aligned}
    \hat S(s,i)&= B(s,i)^\top P(s,i)+ D(s,i)^\top P(s,i)C(s,i)+S(s,i),\\
    \hat R(s,i)&= R(s,i)+D(s,i)^\top P(s,i)D(s,i).
  \end{aligned}
\end{equation*}
It can be found (see Zhang, Li and Xiong \cite{Zhang-Li-Xiong 2018}) that, for the stochastic LQ optimal control problem of Markovian regime switching system,  the existence of a closed-loop optimal strategy implies the existence of an open-loop optimal control, but not vice versa. Thus, there are some LQ problems that are open-loop solvable, but not closed-loop solvable. Such problems cannot be expected to get a regular solution (which does not exist) to the associated GRE \rf{Riccati-0.1}. Therefore, the state feedback representation of the open-loop optimal control might be impossible. To be more convincing, let us look at the following simple example.
\begin{example}\label{11.2.21.47} \rm
Consider the following one-dimensional state equation
$$\left\{\ba{ll}
\ds dX(s)=\big[-\a(s)X(s)+u(s)\big]ds+\sqrt{2\a(s)}X(s)dW(s),\q~s\in[t,1],\\
\ns\ds  X(t)=x,\q~\a(t)=i,
\ea\right.$$
and the nonnegative cost functional
$$J(t,x,i;u(\cd))=\dbE|X(1)|^2.$$
In this example, the GRE reads
(noting that $Q(\cd,i)=0,R(\cd,i)=0,D(\cd,i)=0$ for every $i\in\cS$, and $0^{-1}=0$):
\bel{11.2.22.06}\left\{\ba{ll}
\ds \dot P(s,i)+\sum_{k=1}^D\l_{ik}(s)P(s,k)=0,\q \ae~s\in[t,1],\\
\ns\ds P(T,i)=1, \quad i \in \cS.
\ea\right.\ee
It is not hard to check that GRE \rf{11.2.22.06} has no regular solution (see Section 3 for the definition of regular solution), thus the corresponding LQ problem is not closed-loop solvable. A usual Riccati equation approach specifies the corresponding state feedback control as follows (noting that $Q(\cd,i)=0,R(\cd,i)=0,D(\cd,i)=0$ for every $i\in\cS$, and $0^{-1}=0$):
$$\ba{ll}
\ds u^*(s) \deq -\Big[R(s,i)\!+\!D(s,i)^\top\!P(s,i)D(s,i)\Big]^{-1}
\!\Big[B(s,i)P(s,i)\!+\!D(s,i)^\top\!P(s,i)C(s,i)\!+\!S(s,i)\Big]X(s)\!\equiv\!0,
\ea$$
which is {\it not} an open-loop optimal control for any nonzero initial state $x$. In fact, let $(t,x,i)\in[0,1)\ts\dbR\ts\cS$
be an arbitrary but the fixed initial pair with $x\neq0$. By It\^o's formula,
the state process $X^*(\cd)$ corresponding to $(t,x,i)$ and $u^*(\cd)$ is expressed as
$$X^*(s)=x\cd\exp\left\{-2\int_t^s\a(r)dr+\int_t^s\sqrt{2\a(r)}dW(r)\right\},\q~s\in[t,1].$$
Thus,
$$J(t,x,i;u^*(\cd))=\dbE|X^*(1)|^2=x^2>0.$$
On the other hand, let $\bar{u}(\cd)$ be the control defined by
$$\bar{u}(s)\equiv\frac{x}{t-1}\cd\exp\left\{-2\int_t^s\a(r)dr+\int_t^s\sqrt{2\a(r)}dW(r)\right\},\q~s\in[t,1].$$
By the variation of constants formula, the state process $\bar{X}(\cd)$, corresponding to $(t,x,i)$ and $\bar{u}(\cd)$, can be presented by
$$\ba{ll}
\ds \bar{X}(s)=\exp\left\{-2\int_t^s\a(r)dr+\int_t^s\sqrt{2\a(r)}dW(r)\right\} \\
\ns\ds\qq\q\ \cd\left[x+\int_t^s\exp\left\{2\int_t^r\a(v)dv-\int_t^r\sqrt{2\a(v)}dW(v)\right\}\cd\bar{u}(r)dr\right] \\
\ns\ds\qq\ =\exp\left\{-2\int_t^s\a(r)dr+\int_t^s\sqrt{2\a(r)}dW(r)\right\}\cd\left[x+\frac{s-t}{t-1}x\right],\q~s\in[t,1],
\ea$$
which satisfies $\bar{X}(1)=0$. Hence,
$$J(t,x,i;\bar{u}(\cd))=\dbE|\bar{X}(1)|^2=0<J(t,x,i;u^*(\cd)).$$
Since the cost functional is nonnegative, the open-loop control $\bar{u}(\cd)$ is optimal for the
initial pair $(t,x,i)$, but $u^*(\cd)$ is not optimal.
\end{example}

The above example suggests that the usual solvability of the GRE \rf{Riccati-0.1} no longer helpfully handles the open-loop solvability of certain stochastic LQ problems. It is then natural to ask: {\it When Problem (M-SLQ) is merely open-loop solvable, not closed-loop solvable, is it still possible to get a linear state feedback representation for an open-loop optimal control within the framework of Markovian regime switching system?} The goal of this paper is to tackle this problem.

\ms

The contribution of this paper is to study the weak closed-loop solvability of stochastic LQ optimal control problems with Markovian regime switching system. In detail, we provide an alternative characterization of the open-loop solvability of Problem (M-SLQ) using the perturbation approach adopted in \cite{Sun-Li-Yong 2016}. In order to obtain a linear state feedback representation of open-loop optimal control for Problem (M-SLQ), we introduce the notion of weak closed-loop strategies in the circumstance of stochastic LQ optimal control problem of Markovian regime switching system.
We prove that as long as Problem (M-SLQ) is open-loop solvable, there always exists a weak closed-loop strategy whose outcome actually is an open-loop optimal control. Consequently, the open-loop and weak closed-loop solvability of Problem (M-SLQ) are equivalent on $[0, T)$. Comparing with \cite{Wang-Sun-Yong 2018}, this paper further develops the results in \cite{Wang-Sun-Yong 2018} to the case of stochastic LQ optimal control problems with Markovian regime switching system, which could be applied to financial market models with Markov process, such as interest rate, stocks return and volatility. However, the regime switching jumps will bring some difficulties. For example, the first problem is how to define the closed-loop solvability and weak closed-loop solvability in the circumstance of Markovian regime switching system. The second problem is how to prove the equivalent between the open-loop and weak closed-loop solvability of Problem (M-SLQ) in the circumstance of Markovian regime switching system. We will use the methods of \cite{Sun-Li-Yong 2016}, \cite{Wang-Sun-Yong 2018} and \cite{Zhang-Li-Xiong 2018} to overcome these difficulties.

\ms

The rest of the paper is organized as follows. In Section 2, we collect some preliminary results and introduce a few elementary notions for Problem (M-SLQ). Section 3 is devoted to the study of open-loop solvability by a perturbation method. In section 4, we show how to obtain a weak closed-loop optimal strategy and establish the equivalence between open-loop and weak closed-loop solvability. Finally, an example is presented in Section 5 to illustrate the results we obtained.

\section{Preliminaries}

Throughout this paper, and recall from the previous section, let $(\O,\cF,\dbF,\dbP)$ be a complete filtered probability space on which a standard one-dimensional Brownian motion $W=\{W(t); 0\les t < \i \}$ and a continuous time, finite-state, Markov chain $\a=\{\a(t); 0\les t< \i \}$ are defined, where $\dbF=\{\cF_t\}_{t\ges0}$ is the natural filtration of $W$ and $\a$ augmented by all the $\dbP$-null sets in $\cF$. In the rest of our paper, we will use the following notation:
\begin{eqnarray*}
  \begin{array}{ll}
    \mathbb{R}^n &\q \mbox{the } n\mbox{-dimensional Euclidean space};\\
    M^\top & \q\mbox{the transpose of any vector or matrix } M;\\
    \tr[M] &\q \mbox{the trace of a square matrix } M;\\
    \cR(M) & \q\mbox{the range of the matrix } M;\\
   M^{-1}   &\q \mbox{the Moore-Penrose pseudo-inverse of the matrix } M;\\
   \langle \cd\,,\cd\rangle&\q \mbox{the inner products in possibly different Hilbert spaces};\\
    \mathbb{R}^{n\times m} &\q \mbox{the Euclidean space of all } n\times m \mbox{ real matrices endowed with inner}\\
                &\q\hb {product } \langle M, N\rangle \mapsto\tr[M^\top N] \mbox{ and the norm } |M|=\sqrt{\tr[M^\top M]};\\
    \dbS^n &\q \mbox{the set of all }n\times n \mbox{ symmetric matrices},
  \end{array}
\end{eqnarray*}
and for an $\dbS^n$-valued function $F(\cd)$ on $[t,T]$, we use the
notation $F(\cd)\gg0$ to indicate that $F(\cd)$ is uniformly
positive definite on $[t,T]$, i.e., there exists a constant
$\delta>0$ such that
$$F(s)\ges\delta I,\qq\ae~s\in[t,T].$$
Next, for any $t\in[0,T)$ and Euclidean space $\dbH$, we further introduce the following
spaces of functions and processes:
$$\ba{ll}
C([t,T];\dbH)=\Big\{\f:[t,T]\to\dbH\bigm|\f(\cd)\hb{ is
  continuous }\1n\Big\},\\
\ns\ds
L^p(t,T;\dbH)=\left\{\f:[t,T]\to\dbH\biggm|\int_t^T|\f(s)|^pds<\i\right\},\q1\les p<\i,\\
\ns\ds
L^\infty(t,T;\dbH)=\left\{\f:[t,T]\to\dbH\biggm|\esssup_{s\in[t,T]}|\f(s)|<\i\right\},\ea$$
and
$$\ba{ll}
\ns\ds L^2_{\cF_T}(\O;\dbH)=\Big\{\xi:\O\to\dbH\bigm|\xi\hb{ is
  $\cF_T$-measurable, }\dbE|\xi|^2<\i\Big\},\\
\ns\ds
L_\dbF^2(t,T;\dbH)=\bigg\{\f:[t,T]\times\O\to\dbH\bigm|\f(\cd)\hb{ is
    $\dbF$-progressively measurable, }\dbE\int^T_t|\f(s)|^2ds<\i\bigg\},\\
\ns\ds
L_\dbF^2(\O;C([t,T];\dbH))=\bigg\{\f:[t,T]\times\O\to\dbH\bigm|\f(\cd)\hb{
    is $\dbF$-adapted, continuous, }\dbE\left[\sup_{s\in[t,T]}|\f(s)|^2\right]<\i\bigg\},\\
\ns\ds L^2_\dbF(\O;L^1(t,T;\dbH))=\bigg\{\f:[t,T]\times
  \O\to\dbH\bigm|\f(\cd)\hb{ is $\dbF$-progressively measurable, }
  \dbE\left(\int_t^T|\f(s)|ds\right)^2<\i\bigg\}.\ea$$

Now we start to formulate our system. We identify the state space of the chain $\a$ with a finite set $\cS\deq\{1, 2 \dots, D\}$, where $D\in \mathbb{N}$ and suppose that the chain is homogeneous and irreducible. To specify statistical or probabilistic properties of the chain $\a$, for $t\in[0,\i)$, we define the generator
$\l(t)\deq[\l_{ij}(t)]_{i, j = 1, 2, \dots, D}$ of the chain under $\mathbb{P}$.
This is also called the rate matrix, or the $Q$-matrix. Here, for each $i, j = 1, 2, \dots, D$,
$\l_{ij}(t)$ is the constant transition intensity of the chain from state $i$ to state $j$
at time $t$. Note that $\l_{ij}(t) \ges 0$, for $i \neq j$ and
$\sum^{D}_{j = 1} \l_{ij}(t) = 0$, so $\l_{ii}(t) \les 0$. In what follows for each $i, j = 1,
2, \dots, D$ with $i \neq j$, we suppose that $\l_{ij}(t) > 0$, so
$\l_{ii}(t) < 0$. For each fixed $j = 1, 2, \cdots, D$, let $N_j(t)$ be the number of jumps into state $j$ up to time $t$ and set
$$\l_j (t)\deq \int_0^t\l_{\a(s-)\, j}I_{\{\a (s-)\neq j\}}ds=\sum^{D}_{i = 1, i \neq j}\int^{t}_{0}\l_{ij}(s)I_{\{\a(s-)=i\}} d s.$$
Then for each $j=1,2,\cdots, D$, the term $\widetilde{N}_j (t)\deq N_j(t)-\l_j(t)$ is an $(\dbF, \dbP)$-martingale.

\ms

To guarantee the well-posedness of the state equation \rf{state}, we adopt the following assumption:

\ms

{\bf(H1)} For every $i\in\cS$, the coefficients and nonhomogeneous terms of \rf{state} satisfy
$$\left\{\ba{ll}
\ds A(\cd,i)\in L^1(0,T;\dbR^{n\times n}),
\q  B(\cd,i)\in L^\i(0,T;\dbR^{n\times m}),
\q  b(\cd)\in L^2_\dbF(\Om;L^1(0,T;\dbR^n)),\\
\ns\ds C(\cd,i)\in L^2(0,T;\dbR^{n\times n}),
\q     D(\cd,i)\in L^{\i}(0,T;\dbR^{n\times m}),
\q     \si(\cd)\in L^2_\dbF(0,T;\dbR^n).
\ea\right.$$

The following result, whose proof is similar to the result in \cite[Proposition 2.1]{Sun-Yong 2014}, establishes the well-posedness of the state equation under the assumption (H1).

\bl{Appen1} \sl
Under the assumption (H1), for any initial pair $(t,x,i)\in[0,T)\ts\dbR^n\ts\cS$ and control $u(\cd)\in\cU[t,T]$, the state equation \rf{state} has a unique adapted solution $X(\cd)\equiv X(\cd\ ;t,x,i,u(\cd))$. Moreover, there exists a constant $K>0$, independent of $(t,x,i)$ and $u(\cd)$, such that
\bel{11.1.15.42}\dbE\left[\sup_{t\les s\les T}|X(s)|^2\right]\les
K\dbE\left[|x|^2+\Big(\int_t^T|b(s)|ds\Big)^2+\int_t^T|\si(s)|^2ds+\int_t^T|u(s)|^2ds\right].\ee
\el

To ensure that the random variables in the cost functional \rf{cost} are integrable, we assume the following holds:

\ms

{\bf(H2)} For every $i\in\cS$, the weighting coefficients in the cost functional \rf{cost} satisfy
$$\left\{\ba{ll}
\ds G(i)\in\dbS^n,
\q  Q(\cd,i)\in L^1(0,T;\dbS^n),
\q  S(\cd,i)\in L^2(0,T;\dbR^{m\times n}),
\q  R(\cd,i)\in L^\i(0,T;\dbS^m),\\
\ns\ds \tb{g(i)\in L^2_{\cF_T}(\Om;\dbR^n)},
\q q(\cd,i)\in L^2(0,T;\dbR^n),
\q \rho(\cd,i)\in L^2(0,T;\dbR^m).
\ea\right.$$

\begin{remark}\label{11.11.27.1}\rm Suppose that (H1) holds. Then according to Lemma \ref{Appen1}, for any initial pair
 $(t,x,i)\in[0,T)\ts\dbR^n\ts\cS$ and control $u(\cd)\in\cU[t,T]$, the state equation \rf{state} admits a unique (strong) solution $X(\cd)\equiv X(\cd;t,x,i,u(\cd))$ which belongs to the space $L_\dbF^2(\O;C([t,T];\dbH))$.  In addition, if (H2) holds, then the random variables on the right-hand side of \rf{cost} are integrable, and hence Problem (M-SLQ) is well-posed. \end{remark}

Let us recall some basic notions of stochastic LQ optimal control problems.

\begin{definition}[Open-loop]\label{def-open} \rm
Problem (M-SLQ) is said to be
\begin{enumerate}[~~\,\rm(i)]
\item ({\it uniquely}) {\it open-loop solvable
  for an initial pair $(t,x,i)\in[0,T]\ts\dbR^n\ts\cS$} if there exists a (unique)
  $u^*(\cd)=u^*(\cd\ ;t,x,i)\in\cU[t,T]$ (depending on $(t,x,i)$) such that
  \bel{10.15.17.9}J(t,x,i;u^{*}(\cd))\les J(t,x,i;u(\cd)),\qq \forall u(\cd)\in\cU[t,T].\ee
  Such a $u^*(\cd)$ is called an {\it open-loop optimal control} for $(t,x,i)$.
\item ({\it uniquely}) {\it open-loop solvable} if it is (uniquely) open-loop solvable for all the initial pairs $(t,x,i)\in[0,T]\ts\dbR^n\ts\cS$.
\end{enumerate}
\end{definition}

\begin{definition}[Closed-loop]\label{def-closed}\rm
Let $\Th:[t,T]\ts\cS\ra\dbR^{m\ts n}$ to be deterministic function and $v:[t,T]\ts\Om\ra\dbR^m$ be an $\dbF$-progressively measurable process.
\begin{enumerate}[~~\,\rm(i)]
\item We call $(\Th(\cd,\cd),v(\cd))$ a {\it closed-loop strategy} on $[t,T]$ if
%$\Th(\cd,\a(\cd))\in L^2(t,T;\dbR^{m\ts n})$ and $v(\cd)\in L^2_\dbF(t,T;\dbR^m)$, i.e.,
    %
    \bel{11.17}\dbE\int_t^T|\Th(s,\a(s))|^2ds<\i,\q \hb{and}\q \dbE\int_t^T|v(s)|^2ds<\i.\ee
   The set of all closed-loop strategies $(\Th(\cd,\cd),v(\cd))$ on $[t,T]$ is denoted by $\sC[t,T]$.
\item A closed-loop strategy $(\Th^*(\cd,\cd),v^*(\cd))\in\sC[t,T]$ is said to be {\it optimal} on $[t,T]$ if
  \bel{10.15.17.48}\ba{ll}\ds J(t,x,i;\Th^*(\cd,\alpha(\cd))X^*(\cd)+v^*(\cd))\les J(t,x,i;\Th(\cd,\alpha(\cd))X(\cd)+v(\cd)),\\
  \ns\ds\qq\qq\qq\q\forall(x,i)\in\dbR^n\ts\cS,\q\forall (\Th(\cd,\cd),v(\cd))\in\sC[t,T],\ea\ee
  where $X^*(\cd)$ is the solution to the {\it closed-loop system} under $(\Th^*(\cd,\cd),v^*(\cd))$:
  \bel{10.15.17.58}\left\{\2n\ba{ll}
  \ds dX^*(s)=\Big\{\big[A(s,\a(s))+B(s,\a(s))\Th^*(s,\a(s))\big]X^*(s)+B(s,\a(s))v^*(s)+b(s)\Big\}ds\\
  \ns\ds\qq\q \ ~~~+\Big\{\big[C(s,\a(s))+D(s,\a(s))\Th^*(s,\a(s))\big]X^*(s)\\
  \ns\ds\qq\qq\qq \ +D(s,\a(s))v^*(s)+\si(s)\Big\}dW(s),\qq s\in[t,T],\\
  \ns\ds X^*(t)=x,
  \ea\right.\ee
  and $X(\cd)$ is the solution to the closed-loop system \rf{10.15.17.58} corresponding to $(\Th(\cd,\cd),v(\cd))$.
\item For any $t\in[0,T)$, if a closed-loop optimal strategy (uniquely) exists on $[t,T]$, Problem (M-SLQ) is {\it (uniquely) closed-loop solvable}.
\end{enumerate}
\end{definition}

\begin{remark}\rm
We emphasize that, in the above definition, $\Th$ is a deterministic function, and in \rf{11.17} the randomness of $\Th(\cd,\a(\cd))$ comes from $\a(\cd)$. Moreover, \rf{10.15.17.48} must be true for all $(x,i)\in\dbR^n\ts\cS$. The same remark applies to the definition below.
\end{remark}

\begin{definition}[Weak closed-loop]\label{def-weak-closed}\rm
Let $\Th:[t,T]\ts\cS\ra\dbR^{m\ts n}$ be a deterministic function and $v:[t,T]\ts\Om\ra\dbR^m$ be an $\dbF$-progressively measurable process such that for any $T'\in[t,T)$,
$$\dbE\int_t^{T'}|\Th(s,\a(s))|^2ds<\i,\q\hb{and}\q\dbE\int_t^{T'}|v(s)|^2ds<\i.$$
\begin{enumerate}[~~\,\rm(i)]
\item We call $(\Th(\cd,\cd),v(\cd))$ a {\it weak closed-loop strategy} on $[t,T)$ if for any initial state $(x,i)\in\dbR^n\ts\cS$, the outcome $u(\cd)\equiv\Th(\cd,\a(\cd))X(\cd)+v(\cd)$ belongs to $\cU[t,T]\equiv L^2_{\dbF}(t,T;\dbR^m)$, where $X(\cd)$ is the solution to the {\it weak closed-loop system}:
  \bel{10.16.15.41}\left\{\2n\ba{ll}
  \ds dX(s)=\Big\{\big[A(s,\a(s))+B(s,\a(s))\Th(s,\a(s))\big]X(s)+B(s,\a(s))v(s)+b(s)\Big\}ds\\
  \ns\ds\qq\q~~~+\Big\{\big[C(s,\a(s))+D(s,\a(s))\Th(s,\a(s))\big]X(s)\\
  \ns\ds\qq\qq\qq +D(s,\a(s))v(s)+\si(s)\Big\}dW(s),\qq s\in[t,T],\\
  \ns\ds X(t)=x.
  \ea\right.\ee
  The set of all weak closed-loop strategies is denoted by $\sC_w[t,T]$.
\item A weak closed-loop strategy $(\Th^*(\cd,\cd),v^*(\cd))\in\sC_w[t,T]$ is said to be {\it optimal} on $[t,T)$ if
  \bel{10.15.17.60}\ba{ll}\ds J(t,x,i;\Th^*(\cd,\a(\cd))X^*(\cd)+v^*(\cd))\les J(t,x,i;\Th(\cd,\a(\cd))X(\cd)+v(\cd)),\\
  \ns\ds\qq\qq\qq\q\forall(x,i)\in\dbR^n\ts\cS,\q\forall (\Th(\cd,\cd),v(\cd))\in\sC_w[t,T],\ea\ee
  where $X(\cd)$ is the solution of the closed-loop system \rf{10.16.15.41}, and $X^*(\cd)$ is the solution to the weak closed-loop system \rf{10.16.15.41} corresponding to $(\Th^*(\cd,\cd),v^*(\cd))$.
\item For any $t\in[0,T)$, if a weak closed-loop optimal strategy (uniquely) exists on $[t,T)$, Problem (M-SLQ) is {\it (uniquely) weakly closed-loop solvable}.
\end{enumerate}
\end{definition}

\section{Open-Loop Solvability: A Perturbation Approach}

In this section, we study the open-loop solvability of Problem (M-SLQ) through a perturbation approach.
We begin by assuming that, for any choice of $(t,i)\in[0,T)\ts\cS$,
\bel{10.23.9.53}J^0(t,0,i;u(\cd))\ges0,\q~\forall u(\cd)\in\cU[t,T],\ee
which is necessary for the open-loop solvability of Problem (M-SLQ) according to \cite[Theorem 4.1]{Zhang-Li-Xiong 2018}. In fact, assumption \rf{10.23.9.53} means that $u(\cd)\ra J^0(t,0,i;u(\cd))$ is convex, and one can actually prove that assumption \rf{10.23.9.53} implies the convexity of the mapping $u(\cd)\ra J(t,x,i;u(\cd))$ for any choice of $(t,x,i)\in[0,T)\ts\dbR^n\ts\cS$ (see \cite{Sun-Li-Yong 2016,Zhang-Li-Xiong 2018}).

\ms

For $\e>0$, consider the LQ problem of minimizing the perturbed cost functional
\bel{pertur-cost}\ba{ll}
\ds J_\e(t,x,i;u(\cd))\deq J(t,x,i;u(\cd))+\e\dbE\int_t^T|u(s)|^2ds\\
\ns\ds\qq\qq\qq \ =\dbE\Bigg\{\Big\lan G(\a(T))X(T),X(T)\Big\ran+2\Big\lan g(\a(T)),X(T)\Big\ran\\
\ns\ns\ds\qq\qq\qq\qq \ \ \,+\int_t^T\Bigg[\llan\begin{pmatrix}Q(s,\a(s))&S(s,\a(s))^\top\\S(s,\a(s))&R(s,\a(s))+\e I_m\end{pmatrix}
                                \begin{pmatrix}X(s)\\ u(s)\end{pmatrix},
                                \begin{pmatrix}X(s)\\u(s)\end{pmatrix}\rran \\
\ns\ns\ds\qq\qq\qq\qq\qq\qq+2\llan\begin{pmatrix}q(s,\a(s))\\ \rho(s,\a(s))\end{pmatrix},\begin{pmatrix}X(s)\\ u(s)\end{pmatrix}\rran \Bigg]ds\Bigg\},
\ea\ee
subject to the state equation \rf{state}. We denote this perturbed LQ problem by \bf Problem (M-SLQ)$_\e$ \rm and its value function by $V_\e(\cd,\cd,\cd)$. Notice that the cost functional $J^0_\e(t,x,i;u(\cd))$ of the homogeneous LQ problem associated with Problem (M-SLQ)$_\e$ is
$$J^0_\e(t,x,i;u(\cd))=J^0(t,x,i;u(\cd))+\e\dbE\int_t^T|u(s)|^2ds,$$
which, by \rf{10.23.9.53}, satisfies
$$J^0_\e(t,0,i;u(\cd))\ges\e\dbE\int_t^T|u(s)|^2ds.$$
The Riccati equations associated with Problem (M-SLQ)$_\e$ follow
\bel{Riccati}\left\{\ba{ll}
\ds \dot P_\e(s,i)+P_\e(s,i)A(s,i)+A(s,i)^\top P_\e(s,i)+C(s,i)^\top P_\e(s,i)C(s,i)+Q(s,i)\\
\ns\ds\qq\q\ -\hat{S}_\e(s,i)^\top[\hat{R}_\e(s,i)+\e I_m]^{-1}\hat{S}_\e(s,i)+\sum_{k=1}^D\l_{ik}(s)P_\e(s,k)=0,\q \ae~s\in[0,T],\ i\in\cS,\\
\ns\ds P_\e(T,i)=G(i),
\ea\right.\ee
where for every $(s,i)\in[0,T]\ts\cS$ and $\e>0$,
\begin{equation}
  \label{eq:hatsr}
  \begin{aligned}
    \hat S_\e(s,i)&\deq B(s,i)^\top P_\e(s,i)+ D(s,i)^\top P_\e(s,i)C(s,i)+S(s,i),\\
    \hat R_\e(s,i)&\deq R(s,i)+D(s,i)^\top P_\e(s,i)D(s,i).
  \end{aligned}
\end{equation}
%
%We also denote $\hat S(s,i)$ and $\hat R(s,i)$ by the right hand side of \rf{eq:hatsr} that is independent of $\e$, respectively.
We say that a solution $P_\e(\cd,\cd)\in C([0,T]\times \cS;\dbS^n)$ of (\ref{Riccati}) is said to be {\it regular} if
\begin{align}
      \cR\big(\hat{S}_\e(s,i)\big)&\subseteq\cR\big(\hat{R}_\e(s,i)\big),\q
      \ae~s\in[0,T],\ i\in\cS,\label{11.2.16.29}\\
      \hat{R}_\e(\cd,i)^{-1}\hat{S}_\e(\cd,i)&\in L^2(0,T;\dbR^{m\times n}),\q~i\in\cS,\label{11.2.16.30}\\
      \hat{R}_\e(s,i)&\ges0,\q\ae~s\in[0,T],\ i\in\cS. \label{eq:regular}
\end{align}
A solution $P_\e(\cd,\cd)$ of (\ref{Riccati}) is said to be {\it strongly regular} if
%                                            %
\begin{eqnarray}
    \label{strong-regular}\hat{R}_\e(s,i)\ges \l  I,\q\ae~s\in[0,T],
  \end{eqnarray}
%                                          %
for some $\l>0$. The system of Riccati equations (\ref{Riccati}) is said to be ({\it strongly}) {\it regularly solvable}, if it admits a (strongly) regular solution. Clearly, condition (\ref{strong-regular}) implies \rf{11.2.16.29}-\rf{eq:regular}. Thus, a strongly regular solution $P_\e(\cd,\cd)$ must be regular.
Moreover, it follows from \cite[Theorem 6.3]{Zhang-Li-Xiong 2018} that, under the assumption \rf{10.23.9.53}, Riccati equations \rf{Riccati} have a unique strongly regular solution $P_\e(\cd,\cd)\in C([0,T]\times\cS;\dbS^n)$, and from \rf{eq:regular}, we have
$$\hat R_\e(s,i)+\e I_m\ges\e I_m, \q \ae~s\in[0,T].$$
Furthermore, let $(\eta_\e(\cd),\z_\e(\cd), \xi^\e_1(\cd),\cds,\xi^\e_D(\cd))$ be the adapted solution of the following BSDE:
\bel{eta-zeta-xi}\left\{\2n\ba{ll}
\ds d\eta_\e(s)=-\Big\{\big[A(s,\a(s))+B(s,\a(s))\Th_\e(s,\a(s))\big]^\top\eta_\e(s)\\
\ns\ns\ds \qq\qq\q~+\big[C(s,\a(s))+D(s,\a(s))\Th_\e(s,\a(s))\big]^\top\z_\e(s)\\
\ns\ns\ds \qq\qq\q~+\big[C(s,\a(s))+D(s,\a(s))\Th_\e(s,\a(s))\big]^\top P_\e(s,\a(s))\si(s)\\
\ns\ns\ds \qq\qq\q~+\Th_\e(s,\a(s))^\top \rho(s,\a(s))+P_\e(s,\a(s))b(s)+q(s,\a(s))\Big\}ds\\
\ns\ns\ds \qq\qq+\z_\e(s) dW(s)+\sum_{k=1}^D\xi^\e_k(s)d\wt{N}_k(s),\q s\in[0,T],\\
\ns\ns\ds \eta_\e(T)=g(i),\ea\right.\ee
and let $X_\e(\cd)$ be the solution of the following closed-loop system:
\begin{align}
 \label{eclosed-loop-state}\left\{\2n\ba{ll}
 \ns\ns\ds dX_\e(s)=\Big\{\big[A(s,\a(s))+B(s,\a(s))\Th_\e(s,\a(s))\big]X_\e(s)+B(s,\a(s))v_\e(s)+b(s)\Big\}ds\\
 \ns\ns\ds\qq\qq~+\Big\{\big[C(s,\a(s))+D(s,\a(s))\Th_\e(s,\a(s))\big]X_\e(s)\\
 \ns\ds\qq\qq\qq\ +D(s,\a(s))v_\e(s)+\si(s)\Big\}dW(s),\q s\in[t,T], \\
 \ns\ns\ds X_\e(t)=x,\ea\right.
\end{align}
where $\Th_\e:[0,T]\ts\cS\ra\dbR^{m\ts n}$ and $v_\e:[0,T]\ts\Om\ra\dbR^m$ %recall \rf{eq:hatsr},
are defined by
\begin{align}
  \Th_\e(s,\a(s)) &= - [\hat R_\e(s,\a(s))+\e I_m]^{-1}\hat S_\e(s,\a(s)),\label{10.22.22.37}\\
  v_\e(s) &= - [\hat R_\e(s,\a(s))+\e I_m]^{-1}\hat\rho_\e(s,\a(s)),\label{10.22.22.38}
\end{align}
with
\begin{align}
    \label{eq:hatrrho}
    \hat \rho_\e(s,i)&=B(s,i)^\top\eta_\e(s)+D(s,i)^\top\z_\e(s)+D(s,i)^\top P_\e(s,i)\si(s)+\rho(s,i).
\end{align}
Then from Theorem 5.2 and Corollary 6.5 in Zhang, Li and Xiong \cite{Zhang-Li-Xiong 2018}, the unique open-loop optimal control of Problem (M-SLQ)$_\e$, for the initial pair $(t,x,i)$, is given by
\bel{10.22.22.26}u_\e(s)=\Th_\e(s,\a(s))X_\e(s)+v_\e(s), \q~s\in[t,T].\ee

\begin{remark}\rm
In equation \rf{eta-zeta-xi}, the term of Markovian switching is the sum $\sum\limits^D_{k=1}\xi^\e_k(s)d\tilde{N}_k(s)$, which is equivalent to the sum $\sum\limits^D_{k,l=1}\xi^\e_{kl}(s)d\tilde{N}_{kl}(s)$ as in Donnelly and Heunis \cite{DH12}, and Li and Zheng \cite{LZ15}, because their filtration is the same, where $\tilde{N}_{kl}(s)$ is defined as follows:
$$\tilde{N}_{kk}(t)\equiv0,\qq\tilde{N}_{kl}(t)=N_{kl}(t)-\tilde{\l}_{kl}(t),\q \hb{for } 1\les k,l\les D,\ k\neq l,$$
where $\tilde{N}_{kl}(t)$ is the number of jumps from state $k$ to state $l$ up to time $t$ and
$\tilde{\l}_{kl}(t)=\int_0^t\l_{kl}(s)I_{\{\a(s-)=k\}}ds$.
In fact, for simplicity we could let $ N^1=( N_1,..., N_D)$ and $ N^2=( N_{jk},...)$.
On the one hand, it is easy to know that $\cF^{ N^1}_t\subseteq\cF^{ N^2}_t$.
On the other hand, since
$$\{\a(t)=k\}=\{ \exists s<t,\ \D N^1(r)=0,\ s<r<t,\ \D N_k(s)=1 \},$$
where $\D N_k(s)= N_k(s)- N_k(s-)$, we have $\a(t)\in\cF_t^{ N^1}$. So $N^2(t)\equiv( N_{jk}(t),...)\in \cF^\a_t\subseteq\cF^{ N^1}_t$, which implies that $\cF^{ N^2}_t\subseteq\cF^{ N^1}_t$.
So their filtration is the same.
\end{remark}

Before studying the main result of this section, we prove the following lemma.
\bl{10.22.17.22} \sl
Under Assumptions (H1) and (H2), for any initial pair $(t,x,i)\in[0,T)\ts\dbR^n\ts\cS$, one has
\bel{10.22.17.24}\lim_{\e\da 0}V_\e(t,x,i)= V(t,x,i).\ee
\el

\begin{proof}
Let $(t,x,i)\in[0,T)\ts\dbR^n\ts\cS$ be fixed. On the one hand, for any $\e>0$ and any $u(\cd)\in\cU[t,T]$, we have
$$J_\e(t,x,i;u(\cd))=J(t,x,i;u(\cd))+\e\dbE\int_t^T|u(s)|^2ds\ges J(t,x,i;u(\cd))\ges V(t,x,i).$$
Taking the infimum over all $u(\cd)\in\cU[t,T]$ on the left hand side implies that
\bel{10.22.17.40}V_\e(t,x,i)\ges V(t,x,i).\ee
On the other hand, if $V(t,x,i)$ is finite, then for any $\d>0$, we can find a $u^\d(\cd)\in\cU[t,T]$, independent of $\e>0$, such that
$$J(t,x,i;u^\d(\cd))\les V(t,x,i)+\d.$$
It follows that
$$V_\e(t,x,i)\les J(t,x,i;u^\d(\cd))+\e\dbE\int_t^T|u^\d(s)|^2ds\les V(t,x,i)+\d+\e\dbE\int_t^T|u^\d(s)|^2ds.$$
Letting $\e\ra0$, we obtain
\bel{10.22.17.48}\lim_{\e\da0}V_\e(t,x,i)\les V(t,x,i)+\d.\ee
Since $\d>0$ is arbitrary, by combining \rf{10.22.17.40} and \rf{10.22.17.48}, we obtain \rf{10.22.17.24}. A similar argument applies to the case when $V(t,x,i)=-\i$.
\end{proof}

Now, we present the main result of this section, which provides a characterization of the open-loop solvability of Problem (M-SLQ) in terms of the family $\{u_\e(\cd)\}_{\e>0}$.

\begin{theorem}\label{10.23.16.22} \sl
Let Assumptions (H1)-(H2) and \rf{10.23.9.53} hold. For any given initial pair $(t,x,i)\in[0,T)\ts\dbR^n\ts\cS$,
let $u_\e(\cd)$ be defined by \rf{10.22.22.26}, which is the outcome of the closed-loop optimal strategy
$(\Th_\e(\cd,\cd),v_\e(\cd))$ of Problem (M-SLQ)$_\e$. Then the following statements are equivalent:
\begin{enumerate}[~~\,\rm(i)]
\item Problem (M-SLQ) is open-loop solvable at $(t,x,i)$.
\item The family $\{u_\e(\cd)\}_{\e>0}$ is bounded in $L^2_{\dbF}(t,T;\dbR^m)$, i.e.,
$$\sup_{\e>0}\dbE\int_t^T|u_\e(s)|^2ds<\i.$$
\item The family $\{u_\e(\cd)\}_{\e>0}$ is convergent strongly in $L^2_{\dbF}(t,T;\dbR^m)$ as $\e\ra0$.
\end{enumerate}
\end{theorem}

\begin{proof}
We begin by proving the implication (i) $\Ra$ (ii).
Let $v^*(\cd)$ be an open-loop optimal control of Problem (M-SLQ) for the initial pair $(t,x,i)$. Then for any $\e>0$,
\bel{10.23.17.44}\ba{ll}
\ds V_\e(t,x,i)\les J_\e(t,x,i;v^*(\cd))=J(t,x,i;v^*(\cd))+\e\dbE\int_t^T |v^*(s)|^2ds\\
\ns\ds\qq\qq \ = V(t,x,i)+\e\dbE\int_t^T |v^*(s)|^2ds.
\ea\ee
On the other hand, since $u_\e(\cd)$ is optimal for Problem (M-SLQ)$_\e$ with respect to $(t,x,i)$, we have
\bel{10.23.21.28}\ba{ll}
\ds V_\e(t,x,i)= J_\e(t,x,i;v_\e(\cd))=J(t,x,i;v_\e(\cd))+\e\dbE\int_t^T |v_\e(s)|^2ds\\
\ns\ds\qq\qq \ \ges V(t,x,i)+\e\dbE\int_t^T |v_\e(s)|^2ds.
\ea\ee
Combining \rf{10.23.17.44} and \rf{10.23.21.28} yields that
\bel{10.23.21.30}\dbE\int_t^T|u_\e(s)|^2ds\les\frac{V_\e(t,x,i)-V(t,x,i)}{\e}\les\dbE\int_t^T|v^*(s)|^2ds.\ee
This shows that $\{u_\e(\cd)\}_{\e>0}$ is bounded in $L^2_{\dbF}(t,T;\dbR^m)$.

\ms

For (ii) $\Ra$ (i), the proof is similar to \cite{Wang-Sun-Yong 2018} (See Remark \ref{br-1} below), and
%
%%We next show that (ii) $\Ra$ (i). Since $\{u_\e(\cd)\}_{\e>0}\in L^2_{\dbF}(t,T;\dbR^m)$ is bounded, we can extract a sequence $\{\e_k\}_{k=1}^\i\subseteq(0,\i)$ with $\lim\limits_{k\ra \i}\e_k=0$ such that $\{u_{\e_k}(\cd)\}$ converges
%%weakly to some $u^*(\cd)\in L^2_{\dbF}(t,T;\dbR^m)$. Note that the mapping $u(\cd)\mapsto J(t,x,i;u(\cd))$ is sequentially weakly lower semi-continuous because it is continuous and convex. Then the boundedness of  $\{u_{\e_k}(\cd)\}$, together with \rf{10.22.17.24}, implies that
%%%
%%\bel{10.23.21.50}\ba{ll}
%%%
%%\ds J(t,x,i;u^*(\cd))\les \liminf_{k\ra\i}J(t,x,i;u_{\e_k}(\cd)) \\
%%%
%%\ns\ds\qq\qq\qq \ =\liminf_{k\ra\i}\left[V_{\e_k}(t,x,i)-\e_k\dbE\int_t^T |u_{\e_k}(s)|^2ds\right]\\
%%%
%%\ns\ds\qq\qq\qq \ =V(t,x,i).
%%%
%%\ea\ee
%%%
%%This means that $u^*(\cd)$ is an open-loop optimal control of Problem (M-SLQ) for $(t,x,i)$.
%%%
the implication (iii) $\Ra$ (ii) is trivially true.

\ms

Finally, we prove the implication (ii) $\Ra$  (iii). We divide the proof into two steps.

\ms

{\it Step 1: The family $\{u_\e(\cd)\}_{\e>0}$ converges weakly to an open-loop optimal control of Problem (M-SLQ) for the initial pair $(t,x,i)$ as $\e\ra0$.}

\ms

To verify this, it suffices to show that every weakly convergent subsequence of $\{u_\e(\cd)\}_{\e>0}$ has the same weak limit which is an open-loop optimal control of Problem (M-SLQ) for $(t,x,i)$. Let $u^*_i(\cd)$, $i=1,2$ be the weak limits of two different weakly convergent subsequences $\{u_{i,\e_k}(\cd)\}_{k=1}^\i$ $(i=1,2)$ of $\{u_\e(\cd)\}_{\e>0}$. The same argument as in the proof of (ii) $\Ra$ (i) shows that both $u^*_1(\cd)$ and $u^*_2(\cd)$ are optimal for $(t,x,i)$. Thus, recalling that the mapping $u(\cd)\mapsto J(t, x,i; u(\cd))$ is convex, we have
$$J\left(t,x,i;\frac{u^*_1(\cd)+u^*_2(\cd)}{2}\right)
\les\frac{1}{2}J(t,x,i;u^*_1(\cd))+\frac{1}{2}J(t,x,i;u^*_2(\cd))=V(t,x,i).$$
This means that $\frac{u^*_1(\cd)+u^*_2(\cd)}{2}$ is also optimal for Problem (M-SLQ) with respect to $(t,x,i)$.
Then we can repeat the argument employed in the proof of (i) $\Ra$ (ii), replacing $v^*(\cd)$ by $\frac{u^*_1(\cd)+u^*_2(\cd)}{2}$ to obtain (see \rf{10.23.21.30})
$$\dbE\int_t^T|u_{i,\e_k}(s)|^2ds\les\dbE\int_t^T\Big(\frac{u^*_1(s)+u^*_2(s)}{2}\Big)^2ds,\q~i=1,2.$$
Now, note that
  \begin{align*}
    0\les&~\dbE\int_t^T|u_{i,\e_k}(s)-u_{i}^*(s)|^2ds
     =\dbE\int_t^T\[|u_{i,\e_k}(s)|^2-2\langle u_{i,\e_k}(s),u_{i}^*(s)\rangle+|u_i^*(s)|^2\]ds,
  \end{align*}
which implies that
  \begin{align*}
   2\dbE\int_t^T\langle u_{i,\e_k}(s),u_{i}^*(s)\rangle ds-\dbE\int_t^T|u_i^*(s)|^2ds\les
   \dbE\int_t^T|u_{i,\e_k}(s)|^2ds.
  \end{align*}
Taking inferior limits then yields
  \begin{align*}
   \dbE\int_t^T|u_i^*(s)|^2ds=&~2\liminf_{\e_k\ra0}\dbE\int_t^T\langle u_{i,\e_k}(s),u_{i}^*(s)\rangle ds-\dbE\int_t^T|u_i^*(s)|^2ds\\
   \les&~\liminf_{\e_k\ra0}\dbE\int_t^T|u_{i,\e_k}(s)|^2ds
   \les\dbE\int_t^T\Big(\frac{u^*_1(s)+u^*_2(s)}{2}\Big)^2ds\q~ i=1,2.
  \end{align*}
Adding the above two inequalities and then multiplying by 2, we get
$$2\left[\dbE\int_t^T|u^*_{1}(s)|^2ds+\dbE\int_t^T|u^*_{2}(s)|^2ds\right]
\les\dbE\int_t^T|u^*_1(s)+u^*_2(s)|^2ds,$$
or equivalently (by shifting the integral on the right-hand side to the left-hand side),
$$\dbE\int_t^T|u^*_1(s)-u^*_2(s)|^2ds\les0.$$
It follows that $u^*_1(\cd)=u^*_2(\cd)$, which establishes the claim.

\ms

{\it Step 2: The family $\{u_\e(\cd)\}_{\e>0}$ converges strongly as $\e\ra0$.}

\ms

According to Step 1, the family $\{u_\e(\cd)\}_{\e>0}$ converges weakly to an open-loop optimal
control $u^*(\cd)$ of Problem (M-SLQ) for $(t,x,i)$ as $\e\ra0$. By repeating the argument employed in the proof of (i) $\Ra$ (ii) with $u^*(\cd)$ replacing $v^*(\cd)$, we obtain
\bel{10.24.11.35}\dbE\int_t^T|u_\e(s)|^2ds\les\dbE\int_t^T|u^*(s)|^2ds,\q~\e>0.\ee
On the other hand, since $u^*(\cd)$ is the weak limit of $\{u_\e(\cd)\}_{\e>0}$, we have
\bel{10.24.11.44}\dbE\int_t^T|u^*(s)|^2ds\les\liminf_{\e\ra0}\dbE\int_t^T|u_\e(s)|^2ds.\ee
Combining \rf{10.24.11.35} and \rf{10.24.11.44}, we see that $\dbE\int_t^T|u_\e(s)|^2ds$ actually has the limit
$\dbE\int_t^T|u^*(s)|^2ds$. Therefore (recalling that $\{u_\e(\cd)\}_{\e>0}$ converges weakly to $u^*(\cd)$),
$$\ba{ll}
\ds \lim_{\e\ra0}\dbE\int_t^T|u_\e(s)-u^*(s)|^2ds \\
\ns\ds=\lim_{\e\ra0}\left[\dbE\int_t^T|u_\e(s)|^2ds+\dbE\int_t^T|u^*(s)|^2ds
-2\dbE\int_t^T\langle u^*(s),u_\e(s)\rangle ds\right] \\
\ns\ds0,
\ea$$
which means that $\{u_\e(\cd)\}_{\e>0}$ converges strongly to $u^*(\cd)$ as $\e\ra0$.
\end{proof}

\br{br-1}\rm
A similar result recently appeared in Zhang, Li and Xiong \cite{Zhang-Li-Xiong 2018}, which asserts that if Problem (M-SLQ) is
open-loop solvable at $(t,x,i)$, then the limit of any weakly/strongly convergent subsequence
of $\{u_\e(\cd)\}_{\e>0}$ is an open-loop optimal control for $(t,x,i)$. Our result sharpens that in \cite{Zhang-Li-Xiong 2018} by showing the family $\{u_\e(\cd)\}_{\e>0}$ itself is strongly convergent when Problem (M-SLQ) is open-loop solvable. This improvement has at least two advantages. First, it serves as a crucial bridge to the weak closed-loop solvability presented in the next section. Second, it is much more convenient for computational purposes because subsequence extraction is not required.
\er

\br{br-9}\rm
In Example \ref{11.2.21.47}, since $B=1$ and $D=S=R=0$, we have
\begin{equation*}
  \begin{aligned}
    \hat S_\e(s,i)&\deq B(s,i)^\top P_\e(s,i)+ D(s,i)^\top P_\e(s,i)C(s,i)+S(s,i)=P_\e(s,i) \ \hb{ with }P_\e(T,i)=1,\\
    \hat R_\e(s,i)&\deq R(s,i)+D(s,i)^\top P_\e(s,i)D(s,i)=0.
  \end{aligned}
\end{equation*}
So the condition
$\cR\big(\hat{S}_\e(s,i)\big)\subseteq\cR\big(\hat{R}_\e(s,i)\big),\ \ae~s\in[0,T],\ i\in\cS$ is not satisfied, which implies that GRE \rf{11.2.22.06} has no regular solution.
\er

\section{Weak Closed-Loop Solvability}

In this section, we study the equivalence between open-loop and weak closed-loop solvabilities of Problem (M-SLQ). We shall show that $\Th_\e(\cd,\cd)$ and $v_\e(\cd)$ defined by \rf{10.22.22.37} and \rf{10.22.22.38} converge locally in $[0,T)$, and that the limit pair $(\Th^*(\cd,\cd),v^*(\cd))$ is a weak closed-loop optimal strategy.

\ms

We start with a simple lemma, which enables us to work separately with $\Th_\e(\cd,\cd)$ and $v_\e(\cd)$. Recall that the associated Problem (M-SLQ)$^0$ is to minimize \rf{cost-0} subject to \rf{state-0}.

\bl{lemma-2}\rm
Under Assumptions (H1) and (H2), if Problem (M-SLQ) is open-loop solvable, then so is Problem (M-SLQ)$^0$.
\el

\begin{proof}
For arbitrary $(t,x,i)\in[0,T)\ts\dbR^n\ts\cS$, we note that if $b(\cd,\cd),\si(\cd,\cd),g(\cd),q(\cd,\cd),\rho(\cd,\cd)=0$, then the adapted solution
$(\eta_\e(\cd),\z_\e(\cd), \xi^\e_1(\cd),\cds,\xi^\e_D(\cd))$ to BSDE \rf{eta-zeta-xi} is identically zero, and hence the process $v_\e(\cd)$ defined by \rf{10.22.22.38} is also identically zero.
By Theorem \ref{10.23.16.22}, to prove that Problem (M-SLQ)$^0$ is open-loop solvable at $(t,x,i)$, we need to verify that the family $\{u_\e(\cd)\}_{\e>0}$ is bounded in $L^2_{\dbF}(t,T;\dbR^m)$, where (see \rf{10.22.22.26} and note that $v_\e(\cd)=0$),
\bel{10.26.16.16}u_\e(\cd)=\Th_\e(\cd,\a(\cd))X_\e(\cd),\ee
with $X_\e(\cd)$ is the solution to the following equation:
\begin{align}
 \label{eclosed-loop-state-2}\left\{\2n\ba{ll}
 \ns\ns\ds dX_\e(s)=\big[A(s,\a(s))+B(s,\a(s))\Th_\e(s,\a(s))\big]X_\e(s)ds\\
 \ns\ns\ds\qq\qq~ +\big[C(s,\a(s))+D(s,\a(s))\Th_\e(s,\a(s))\big]X_\e(s)dW(s),\q s\in[t,T], \\
 \ns\ns\ds X_\e(t)=x.\ea\right.
\end{align}
To this end, we return to Problem (M-SLQ). Let $v_\e(\cd)$ be defined in \rf{10.22.22.38} and denote by
$X_\e( \cd\ ;t,x,i)$ and $X_\e( \cd\ ;t,0,i)$ solutions to \rf{eclosed-loop-state} with respect to the initial pairs $(t,x,i)$ and $(t,0,i)$, respectively. Since Problem (M-SLQ) is open-loop solvable at both $(t,x,i)$ and $(t,0,i)$, by Theorem \ref{10.23.16.22}, the families
\bel{10.26.16.19}\ba{ll}
\ds u_\e(s;t,x,i)\deq\Th_\e(s,\a(s))X_\e(s;t,x,i)+v_\e(s),\\
\ns\ds u_\e(s;t,0,i)\deq\Th_\e(s,\a(s))X_\e(s;t,0,i)+v_\e(s),
\ea\q~s\in[t,T],\ee
are bounded in $L^2_{\dbF}(t,T;\dbR^m)$. Note that due to that the process $v_\e(\cd)$ is independent of the
initial state, the difference $X_\e( \cd\ ;t,x,i)-X_\e( \cd\ ;t,0,i)$ also satisfies the same equation \rf{eclosed-loop-state-2}. Then by the uniqueness of adapted solutions of SDEs, we obtain that
$$X_\e(\cd)=X_\e( \cd\ ;t,x,i)-X_\e( \cd\ ;t,0,i),$$
which, combining \rf{10.26.16.16} and \rf{10.26.16.19}, implies that
$$u_\e(\cd)=u_\e( \cd ,t,x,i)-u_\e( \cd ,t,0,i).$$
Since $\{u_\e( \cd ,t,x,i)\}_{\e>0}$ and  $\{u_\e( \cd ,t,0,i)\}_{\e>0}$ are bounded in
$L^2_{\dbF}(t,T;\dbR^m)$, so is $\{u_\e(\cd)\}_{\e>0}$. Finally, it follows from Theorem \ref{10.23.16.22} that
Problem (M-SLQ)$^0$ is open-loop solvable.
\end{proof}

Next, we prove that the family $\{\Th_\e(\cd,\cd)\}_{\e>0}$ defined by \rf{10.22.22.37} is locally convergent in $[0,T)$.

\bp{10.26.17.22} \sl
Let (H1) and (H2) hold. Suppose that Problem (M-SLQ)$^0$ is open-loop solvable. Then the family $\{\Th_\e(\cd,\cd)\}_{\e>0}$ defined by \rf{10.22.22.37} converges in $L^2(0,T';\dbR^{m\ts n})$ for any $0<T'<T$;
that is, there exists a locally square-integrable deterministic function $\Th^*:[0,T)\ts\cS\ra\dbR^{m\ts n}$ such that
$$\lim_{\e\ra0}\dbE\int_0^{T'}|\Th_\e(s,\a(s))-\Th^*(s,\a(s))|^2ds=0,\q~\forall~0<T'<T.$$
\ep

\begin{proof}
We need to show that for any $0<T'<T$, the family $\{\Th_\e(\cd)\}_{\e>0}$ is Cauchy in  $L^2(0,T';\dbR^{m\ts n})$.
To this end, let us first fix an arbitrary initial $(t,i)\in[0,T)\ts\cS$ and let
$\Phi_\e(\cd)\in L^2_{\dbF}(\Om;C([t,T];\dbR^{n\ts n}))$ be the solution to the following SDE: %for $\dbR^{n\ts n}$-valued processes
\begin{align}
 \label{eclosed-loop-state-3}\left\{\2n\ba{ll}
 \ns\ns\ds d\Phi_\e(s)=\big[A(s,\a(s))+B(s,\a(s))\Th_\e(s,\a(s))\big]\Phi_\e(s)ds\\
 \ns\ns\ds\qq\qq\ +\big[C(s,\a(s))+D(s,\a(s))\Th_\e(s,\a(s))\big]\Phi_\e(s)dW(s),\q~s\in[t,T], \\
 \ns\ns\ds \Phi_\e(t)=I_n.\ea\right.
\end{align}
Clearly, for any initial state $x$, from the uniqueness of SDEs,
the solution of \rf{eclosed-loop-state-2} is given by
$$X_\e(s)=\Phi_\e(s)x,\q~s\in[t,T].$$
Since Problem (M-SLQ)$^0$ is open-loop solvable, by Theorem \ref{10.23.16.22}, the family
$$u_\e(s)=\Th_\e(s,\a(s))X_\e(s)=\Th_\e(s,\a(s))\Phi_\e(s)x,\q~s\in[t,T],\q \e>0$$
is strongly convergent in $L^2_\dbF(t,T;\dbR^{m})$ for any $x\in\dbR^n$. It follows that $\{\Th_\e(\cd,\cd)\Phi_\e(\cd)\}_{\e>0}$ converges strongly in $L^2_\dbF(t,T;\dbR^{m\ts n})$ as $\e\ra0$.
Denote $U_\e(\cd)=\Th_\e(\cd,\cd)\Phi_\e(\cd)$ and let $U^*(\cd)$ be the strong limit of $U_\e(\cd)$.
By Jensen's inequality, we get
\bel{10.27.10.51}\int_t^T\big|\dbE[U_\e(s)]-\dbE[U^*(s)]\big|^2ds
\les\dbE\int_t^T\big|U_\e(s)-U^*(s)\big|^2ds\ra0\q\hb{as}\q\e\ra0.\ee
Moreover, from \rf{eclosed-loop-state-3}, one see that $\dbE^\a[\Phi_\e(\cd)]$ satisfies the following ODE:
$$\left\{\2n\ba{ll}
 \ds d\dbE^\a_s[\Phi_\e(s)]=\Big\{A(s,\a(s))\mathbb{E}^\a_s[\Phi_\e(s)]+B(s,\a(s))\dbE^\a_s[U_\e(s)]\Big\}ds,\q~s\in[t,T],\\
 \ns\ds \dbE^\a_t[\Phi_\e(t)]=I_n.
 \ea\right.$$
By the standard results of ODE, combining \rf{10.27.10.51}, the family of continuous functions $\dbE^\a[\Phi_\e(\cd)]$ converges uniformly to the solution of
$$\left\{\2n\ba{ll}
 \ds d\dbE^\a_s[\Phi^*(s)]=\Big\{A(s,\a(s))\mathbb{E}^\a_s[\Phi^*(s)]+B(s,\a(s))\dbE^\a_s[U^*(s)]\Big\}ds,\q~s\in[t,T],\\
 \ns\ds \dbE^\a_t[\Phi^*(t)]=I_n.
 \ea\right.$$
Thus, by noting that $\dbE^\a_t[\Phi^*(t)]=I_n$, we can choose some small constant $\D_t>0$ such that for small $\e>0$,
\begin{enumerate}[~~\,\rm(i)]
\item [(a)] $\dbE^\a_s[\Phi_\e(s)]$ is invertible for all $s\in[t,t+\D_t]$, and
\item [(b)] $|\dbE^\a_s[\Phi_\e(s)]|\ges\frac{1}{2}$ for all $s\in[t,t+\D_t]$.
\end{enumerate}
We claim that the family $\{\Th_\e(\cd,i)\}_{\e>0}$ is Cauchy in $L^2(t,t+\D_t;\dbR^{m\ts n})$ for each $i\in\cS$.
Indeed, first note that when $s\in[t,t+\D_t]$, note that (a) and (b), one has
\begin{align*}
U_\e(s)=\Th_\e(s,\a(s))\Phi_\e(s)~&\Longrightarrow~\dbE^\a_s[U_\e(s)]=\Th_\e(s,\a(s))\dbE^\a_s[\Phi_\e(s)]\\
~&\Longrightarrow~\Th_\e(s,\a(s))=\dbE^\a_s[U_\e(s)]\dbE^\a_s[\Phi_\e(s)]^{-1}.
\end{align*}
Then we have
$$\ba{ll}
\ds\dbE\int_t^{t+\D_t}\big|\Th_{\e_1}(s,\a(s))-\Th_{\e_2}(s,\a(s))\big|^2ds\\
\ns\ds=\dbE\int_t^{t+\D_t}\Big|\dbE^\a_s[U_{\e_1}(s)]\dbE^\a_s[\Phi_{\e_1}(s)]^{-1}
-\dbE[U_{\e_2}(s)]\dbE^\a_s[\Phi_{\e_2}(s)]^{-1}\Big|^2ds\\
\ns\ds\les2\dbE\int_t^{t+\D_t}\big|\dbE^\a_s[U_{\e_1}(s)-U_{\e_2}(s)]\big|^2\cd\big|
\dbE^\a_s[\Phi_{\e_1}(s)]^{-1}\big|^2ds\\
\ns\ds\q+2\dbE\int_t^{t+\D_t}\big|\dbE^\a_s[U_{\e_2}(s)]\big|^2\cd\big|\dbE^\a_s[\Phi_{\e_1}(s)]^{-1}
-\dbE^\a_s[\Phi_{\e_2}(s)]^{-1}\big|^2ds\\
\ns\ds=2\dbE\int_t^{t+\D_t}\big|\dbE^\a_s[U_{\e_1}(s)-U_{\e_2}(s)]\big|^2\cd\big|\dbE^\a_s[\Phi_{\e_1}(s)]^{-1}\big|^2ds\\
\ns\ds\q+2\dbE\int_t^{t+\D_t}\big|\dbE^\a_s[U_{\e_2}(s)]\big|^2\cd\big|\dbE^\a_s[\Phi_{\e_1}(s)-\Phi_{\e_2}(s)]\big|^2
\cd\big|\dbE^\a_s[\Phi_{\e_1}(s)]^{-1}\big|^2\cd\big|\dbE^\a_s[\Phi_{\e_2}(s)]^{-1}\big|^2ds\\
\ns\ds\les8\int_t^{t+\D_t}\big|\dbE[U_{\e_1}(s)-U_{\e_2}(s)]\big|^2ds
+32\int_t^{t+\D_t}\big|\dbE[U_{\e_2}(s)]\big|^2ds\cd
\Big(\sup_{t\les s\les t+\D_t}\big|\dbE[\Phi_{\e_1}(s)]-\dbE[\Phi_{\e_2}(s)]\big|^2\Big).
\ea$$
Since $\{U_\e(\cd)\}_{\e>0}$ is Cauchy in $L^2_{\dbF}(t,T;\dbR^{m\ts n})$ and $\{\dbE[\Phi_\e(\cd)]\}_{\e>0}$ converges uniformly on $[t,T]$, the last two terms of the above inequality approach to zero as $\e_1,\e_2\ra0$,
which implies that  $\{\Th_\e(\cd,i)\}_{\e>0}$ is Cauchy in $L^2(t,t+\D_t;\dbR^{m\ts n})$ for each $i\in\cS$.

\ms

Next we use a compactness argument to prove that, for each $i\in\cS$, $\{\Th_\e(\cd,i)\}_{\e>0}$ is actually Cauchy in
$L^2(0,T';\dbR^{m\ts n})$ for any $0<T'<T$. Take any $T'\in(0,T)$. From the preceding argument we see that
for each $t\in[0,T']$, there exists a small $\D_t>0$ such that $\{\Th_\e(\cd,i)\}_{\e>0}$ is Cauchy in $L^2(t,t+\D_t;\dbR^{m\ts n})$. Since $[0,T']$ is compact, we can choose finitely many $t\in[0,T']$, say,
$t_1,t_2,...,t_k,$ such that  $\{\Th_\e(\cd,i)\}_{\e>0}$ is Cauchy in each $L^2(t_j,t_j+\D_{t_j};\dbR^{m\ts n})$
and $[0,T']\subseteq\bigcup_{j=1}^k[t_j,t_j+\D_{t_j}]$. It follows that
$$\ba{ll}
\ds \dbE\int_t^T\big|\Th_{\e_1}(s,\a(s))-\Th_{\e_2}(s,\a(s))\big|^2ds\\
\ns\ds\les\sum_{j=1}^k\dbE\int_t^{t_j+\D_{t_j}}\big|\Th_{\e_1}(s,\a(s))-\Th_{\e_2}(s,\a(s))\big|^2ds
\ra0\q\hb{as}\q\e_1,\e_2\ra0.\ea$$
The proof is therefore completed.
\end{proof}

The following result shows that the family $\{v_\e(\cd)\}_{\e>0}$ defined by \rf{10.22.22.38} is also locally convergent in $[0,T)$.

\bp{11.1.11.08} \sl
Let (H1) and (H2) hold. Suppose that Problem (M-SLQ) is open-loop solvable. Then the family $\{v_\e(\cd)\}_{\e>0}$ defined by \rf{10.22.22.38} converges in $L^2(0,T';\dbR^{m})$ for any $0<T'<T$; that is, there exists a locally square-integrable deterministic function $v^*(\cd):[0,T)\ra\dbR^{m}$ such that
$$\lim_{\e\ra0}\dbE\int_0^{T'}|v_\e(s)-v^*(s)|^2ds=0,\q~\forall~0<T'<T.$$
\ep

\begin{proof} \rm
Let $X_\e(s)$, $0\les s\les T$, be the solution to the closed-loop system \rf{eclosed-loop-state} with respect
to initial time $t=0$. Then, on the one hand, from the linearity of the state equation \rf{state} and Lemma \ref{Appen1}, we have
$$\dbE\left[\sup_{0\les s\les T}|X_{\e_1}(s)-X_{\e_2}(s)|^2\right]
\les K\dbE\int_0^T|u_{\e_1}(s)-u_{\e_2}(s)|^2ds.$$
On the other hand, since Problem (M-SLQ) is open-loop solvable, Theorem \ref{10.23.16.22} implies that the family
\bel{11.1.16.50}u_\e(s)=\Th_\e(s,\a(s))X_\e(s)+v_\e(s),\q~s\in[0,T];\q~\e>0\ee
is Cauchy in $L^2_{\dbF}(0,T;\dbR^m)$, i.e.,
\bel{11.1.16.51}\dbE\int_0^T|u_{\e_1}(s)-u_{\e_2}(s)|^2ds\ra0\q\hb{as}\q\e_1,\e_2\ra0.\ee
Therefore
\bel{11.1.16.28}\dbE\left[\sup_{0\les s\les T}|X_{\e_1}(s)-X_{\e_2}(s)|^2\right]\ra0\q\hb{as}\q\e_1,\e_2\ra0.\ee
Now for every $0<T'<T$. Since Problem (M-SLQ) is open-loop solvable, according to
Lemma \ref{lemma-2} and Proposition \ref{10.26.17.22}, the family $\{\Th_\e(\cd,i)\}_{\e>0}$ is Cauchy in
$L^2(0,T';\dbR^{m\ts n})$ for every $i\in\cS$. Thus, combining \rf{11.1.16.28}, we have
$$\ba{ll}
\ds \dbE\int_0^{T'}\Big|\Th_{\e_1}(s,\a(s))X_{\e_1}(s)-\Th_{\e_2}(s,\a(s))X_{\e_2}(s)\Big|^2ds\\
%
%\ns\ds\les2\dbE\int_0^{T'}|\Th_{\e_1}(s,\a(s))-\Th_{\e_2}(s,\a(s))|^2\cd|X_{\e_1}(s)|^2ds
%+2\dbE\int_0^{T'}|\Th_{\e_2}(s,\a(s))|^2\cd|X_{\e_1}(s)-X_{\e_2}(s)|^2ds\\
%
\ns\ds\les2\dbE\int_0^{T'}|\Th_{\e_1}(s,\a(s))-\Th_{\e_2}(s,\a(s))|^2ds\cd\dbE\left[\sup_{0\les s\les T'}|X_{\e_1}(s)|^2\right]\\
\ns\ds\q+2\dbE\int_0^{T'}|\Th_{\e_2}(s,\a(s))|^2ds\cd\dbE\left[\sup_{0\les s\les T'}|X_{\e_1}(s)-X_{\e_2}(s)|^2\right]\\
\ns\ds\longrightarrow0\q\hb{as}\q\e_1,\e_2\ra0,
\ea$$
which combing \rf{11.1.16.50} and \rf{11.1.16.51}, implies that
$$\ba{ll}
\ds \dbE\int_0^{T'}|v_{\e_1}(s)-v_{\e_2}(s)|^2ds\\
\ns\ds=\dbE\int_0^{T'}\Big|[u_{\e_1}(s)-\Th_{\e_1}(s,\a(s))X_{\e_1}(s)]-[u_{\e_2}(s)
-\Th_{\e_2}(s,\a(s))X_{\e_2}(s)]\Big|^2ds\\
\ns\ds\les2\dbE\int_0^{T'}|u_{\e_1}(s)-u_{\e_2}(s)|^2ds
+2\dbE\int_0^{T'}|\Th_{\e_1}(s,\a(s))X_{\e_1}(s)-\Th_{\e_2}(s)X_{\e_2}(s,\a(s))|^2ds\\
\ns\ns\ds\longrightarrow0\q\hb{as}\q\e_1,\e_2\ra0.
\ea$$
This shows that the family $\{v_\e(\cd)\}_{\e>0}$ converges in $L^2_{\dbF}(0,T';\dbR^m)$.
\end{proof}

We are now ready to state and prove the main result of this section, which establishes
the equivalence between open-loop and weak closed-loop solvability of Problem (M-SLQ).

\bt{main-result} \sl
Let (H1) and (H2) hold. If Problem (M-SLQ) is open-loop solvable, then the
limit pair $(\Th^*(\cd,\cd),v^*(\cd))$ obtained in Propositions \ref{10.26.17.22} and \ref{11.1.11.08} is a weak closed-loop optimal strategy of Problem (M-SLQ) on any $[t,T)$. Consequently, the open-loop and weak closed-loop
solvability of Problem (M-SLQ) are equivalent.
\et

\begin{proof}
From Definition \ref{def-weak-closed}, it is obvious that the weak closed-loop solvability of Problem (M-SLQ) implies the open-loop  solvability of Problem (M-SLQ). In the following, we consider the inverse case.

\ms

Take an arbitrary initial pair $(t,x,i)\in[0,T)\ts\dbR^n\ts\cS$ and let $\{u_\e(s);t\les s\les T\}_{\e>0}$ be the family defined by \rf{10.22.22.26}. Since Problem (M-SLQ) is open-loop solvable at $(t,x,i)$, by
Theorem \ref{10.23.16.22}, $\{u_\e(s);t\les s\les T\}_{\e>0}$ converges strongly to an open-loop optimal control
$\{u^*(s);t\les s\les T\}_{\e>0}$ of Problem (M-SLQ) (for the initial pair $(t,x,i)$). Let
$\{X^*(s);t\les s\les T\}_{\e>0}$ be the corresponding optimal state process; i.e., $X^*(\cdot)$ is the adapted solution of the following equation:
$$\left\{\2n\ba{ll}
 \ns\ns\ds dX^*(s)=\Big[A(s,\a(s))X^*(s)+B(s,\a(s))u^*(s)+b(s)\Big]ds\\
 \ns\ns\ds\qq\qq+\Big[C(s,\a(s))X^*(s)+D(s,\a(s))u^*(s)+\si(s)\Big]dW(s),\q~s\in[t,T], \\
 \ns\ns\ds X^*(t)=x.\ea\right.$$
If we can show that
\bel{11.1.21.43}u^*(s)=\Th^*(s,\a(s))X^*(s)+v^*(s),\q~t\les s<T,\ee
then $(\Th^*(\cd,\cd),v^*(\cd))$ is clearly a weak closed-loop optimal strategy of Problem (M-SLQ) on $[t,T)$. To justify the argument, we note first that by Lemma \ref{Appen1}, we obtain
$$\dbE\left[\sup_{t\les s\les T}|X_\e(s)-X^*(s)|^2\right]\les
K\dbE\int_t^T|u_\e(s)-u^*(s)|^2ds\ra0\q\hb{as}\q\e\ra0,$$
where $\{X_\e(s);t\les s\les T\}_{\e>0}$ is the solution of equation \rf{eclosed-loop-state}.
Second, by Propositions \ref{10.26.17.22} and \ref{11.1.11.08}, one has
$$\left\{\ba{ll}
\ds \lim_{\e\ra0}\dbE\int_0^{T'}|\Th_\e(s,\a(s))-\Th^*(s,\a(s))|^2ds=0,\q~\forall0<T'<T,\\
\ns\ds \lim_{\e\ra0}\dbE\int_0^{T'}|v_\e(s)-v^*(s)|^2ds=0,\q~\forall0<T'<T.
\ea\right.$$
It follows that for any $0<T'<T$,
$$\ba{ll}
\ds \dbE\int_0^{T'}\Big|\big[\Th_\e(s,\a(s))X_\e(s)+v_\e(s)\big]-\big[\Th^*(s,\a(s))X^*(s)+v^*(s)\big]\Big|^2ds\\
\ns\ds\les 2\dbE\int_0^{T'}|\Th_\e(s,\a(s))X_\e(s)-\Th^*(s,\a(s))X^*(s)|^2ds+2\dbE\int_0^{T'}|v_\e(s)-v^*(s)|^2ds\\
%
%\ns\ds\les4\dbE\int_0^{T'}|\Th_\e(s,\a(s))|^2\cd|X_\e(s)-X^*(s)|^2ds
%+4\dbE\int_0^{T'}|\Th_\e(s,\a(s))-\Th^*(s,\a(s))|^2\cd|X^*(s)|^2ds\\
%
%\ns\ds\q+2\dbE\int_0^{T'}|v_\e(s)-v^*(s)|^2ds\\
%
\ns\ds\les4\dbE\int_0^{T'}|\Th_{\e}(s,\a(s))|^2ds\cd\dbE\left[\sup_{0\les s\les T'}|X_{\e}(s)-X^*(s)|^2\right]
+2\dbE\int_0^{T'}|v_\e(s)-v^*(s)|^2ds\\
\ns\ds\q+4\dbE\int_0^{T'}|\Th_{\e}(s,\a(s))-\Th^*(s,\a(s))|^2ds\cd\dbE\left[\sup_{0\les s\les T'}|X^*(s)|^2\right]\\
\ns\ns\ds \longrightarrow0\q\hb{as}\q\e\ra0.
\ea$$
Recall that $u_\e(s)=\Th_\e(s,\a(s))X_\e(s)+v_\e(s)$ converges strongly to $u_\e^*(s)$, $t\les s\les T$,
in $L^2_\dbF(t,T;\dbR^m)$ as $\e\ra0$. Thus, \rf{11.1.21.43} must hold. The above argument shows that the
open-loop solvability implies the weak closed-loop solvability. Consequently, the open-loop and weak closed-loop
solvability of Problem (M-SLQ) are equivalent. This completes the proof.
\end{proof}

\section{Examples}

There are some (M-SLQ) problems that are open-loop solvable, but not closed-loop solvable; for such problems, one could not expect to get a regular solution (which does not exist) to the associated GRE \rf{Riccati}, so that the state feedback representation of the open-loop optimal control might be impossible. In fact, Example \ref{11.2.21.47} has illustrated this conclusion. However, Theorem \ref{main-result} shows that the open-loop and weak closed-loop solvability of Problem (M-SLQ) are equivalent. In the following, we present another example to illustrate the procedure for finding weak closed-loop optimal strategies for some (M-SLQ) problems that are open-loop solvable (and hence weakly closed-loop solvable) but not closed-loop solvable.

\begin{example}\label{11.2.21.59} \rm
In order to present the procedure more clearly, we simplify the problem. Let $T=1$ and $D=2$, that is, the state space of $\a(\cd)$ is $\cS=\{1,2\}$.
For the generator $\l(s)\deq [\l_{ij}(s)]_{i, j = 1, 2}$, note that $\sum^{2}_{j = 1} \l_{ij}(s) = 0$ for $i\in\cS$, then
$$\l(s)=\begin{pmatrix}\l_{11}(s)&\l_{12}(s)\\\l_{21}(s)&\l_{22}(s)\end{pmatrix}
=\begin{pmatrix}\l_{11}(s)&-\l_{11}(s)\\-\l_{22}(s)&\l_{22}(s)\end{pmatrix},\q~s\in[0,1].$$
Consider the following Problem (M-SLQ) with one-dimensional state equation
\bel{11.5.21.27}\left\{\ba{ll}
\ns\ds dX(s)=\Big[-\a(s)X(s)+u(s)+b(s)\Big]ds+\sqrt{2\a(s)}X(s)dW(s),\q~s\in[t,1],\\
\ns\ns\ds  X(t)=x,\q~\a(t)=i,
\ea\right.\ee
and the cost functional
$$J(t,x,i;u(\cd))=\dbE|X(1)|^2,$$
where the nonhomogeneous term $b(\cd,\cd)$ is given by
$$b(s)=\left\{\ba{ll}
\ds \frac{1}{\sqrt{1-s}}\cd\exp\left\{\int_0^s\sqrt{2\a(r)}dW(r)-2\int_0^s\a(r)dr\right\},\qq \mbox{if }s\in[0,1);\\
\ns\ns\ns\ds 0,\qq \mbox{if } s=1.
\ea\right.$$
It is easy to see that $b(\cd,i)\in L^2_\dbF(\Om;L^1(0,1;\dbR))$ for each $i\in\cS$. In fact,
$$\ba{ll}
\ds\qq\ \ \dbE\left(\int_0^1|b(s)|ds\right)^2
=\dbE\left(\int_0^1\frac{1}{\sqrt{1-s}}\cd\exp\left\{\int_0^s\sqrt{2\a(r)}dW(r)-2\int_0^s\a(r)dr\right\}ds\right)^2\\
\ns\ns\ds\qq\qq\qq\qq\qq\ \les
\dbE\left(\int_0^1\frac{1}{\sqrt{1-s}}\cd\exp\left\{\int_0^s\sqrt{2\a(r)}dW(r)-\int_0^s\a(r)dr\right\}ds\right)^2\\
\ns\ns\ds\qq\qq\qq\qq\qq\ \les
\dbE\left(\int_0^1\frac{1}{\sqrt{1-s}}ds\cd
\sup_{0\les s\les 1}\exp\left\{\int_0^s\sqrt{2\a(r)}dW(r)-\int_0^s\a(r)dr\right\}\right)^2\\
\ns\ns\ds\qq\qq\qq\qq\qq\ =\left(\int_0^1\frac{1}{\sqrt{1-s}}ds\right)^2\cd
\dbE\left(\sup_{0\les s\les 1}\exp\left\{\int_0^s\sqrt{2\a(r)}dW(r)-\int_0^s\a(r)dr\right\}\right)^2\\
\ns\ns\ds\qq\qq\qq\qq\qq\ =4\
\dbE\left(\sup_{0\les s\les 1}\exp\left\{\int_0^s\sqrt{2\a(r)}dW(r)-\int_0^s\a(r)dr\right\}\right)^2.
\ea$$
Since the term $\exp\left\{\int_0^s\sqrt{2\a(r)}dW(r)-\int_0^s\a(r)dr\right\}$ is a square-integrable martingale, note that $\a(\cd)$ belongs to $\cS=\{1,2\}$, it follows from Doob's maximal inequality that
$$\ba{ll}
\ds \dbE\left(\sup_{0\les s\les 1}\exp\left\{\int_0^s\sqrt{2\a(r)}dW(r)-\int_0^s\a(r)dr\right\}\right)^2
\les 4\dbE\exp\left\{2\int_0^1\sqrt{2\a(r)}dW(r)-2\int_0^1\a(r)dr\right\}\\
\ns\ds\qq\qq\qq\qq\qq\qq\qq\qq\qq\qq\qq\ %\les 4\exp\left\{2\int_0^1\a(r)dr\right\}
\les4e^4.
\ea$$
Thus,
$$ \dbE\left(\int_0^1|b(s)|ds\right)^2\les 16e^4,$$
which implies that $b(\cd,i)\in L^2_\dbF(\Om;L^1(0,1;\dbR))$ for each $i\in\cS$.

\ms

\ms

We first claim that this (M-SLQ) problem is not closed-loop solvable on any $[t,1]$. Indeed, the generalized Riccati equation associate with this problem reads
$$\left\{\ba{ll}
\ds \dot P(s,1)+\l_{11}(s)P(s,1)-\l_{11}(s)P(s,2)=0,\q~\ae~s\in[t,1],\q~\\
\ns\ds P(1,1)=1,
\ea\right. \mbox{for } i=1,
$$
and
$$\left\{\ba{ll}
\ds \dot P(s,2)-\l_{22}(s)P(s,1)+\l_{22}(s)P(s,2)=0,\q \ae~s\in[t,1],\q~\\
\ns\ds P(1,2)=1,
\ea\right. \mbox{for } i=2,
$$
whose solutions are $P(s,1)=P(s,2)=1$, or $P(s,i)\equiv1,$ for $(s,i)\in[0,1]\ts\cS.$
Then for any $s\in[t,1]$ and $i\in\cS$, we have
$$\ba{ll}
\ds \cR\big(\hat{S}(s,i)\big)=\cR(1)=\dbR,\\
\ns\ds \cR\big(\hat{R}(s,i)\big)=\cR(0)=\{0\},\q~
\ea\Longrightarrow\q~\cR\big(\hat{S}(s,i)\big)\nsubseteq\cR\big(\hat{R}(s,i)\big).$$
where
\begin{equation}\label{eq:hatsr1}
  \begin{aligned}
    \hat S(s,i)&\deq B(s,i)^\top P(s,i)+ D(s,i)^\top P(s,i)C(s,i)+S(s,i), \\
    \hat R(s,i)&\deq R(s,i)+D(s,i)^\top P(s,i)D(s,i).
  \end{aligned}
\end{equation}
Therefore, the range inclusion condition is not satisfied. This implies that our claim holds.

\ms

In the following, we use Theorem \ref{10.23.16.22} to conclude that the above (M-SLQ) problem is open-loop solvable
(and hence, by Theorem \ref{main-result}, weakly closed-loop solvable). Without loss of generality, we consider only the open-loop solvability at $t=0$. To this end, let $\e>0$ be arbitrary and consider Riccati equations \rf{Riccati}, which, in our example, read:
$$\left\{\ba{ll}
\ds \dot P_\e(s,1)-\frac{1}{\e}P_\e(s,1)^2+\l_{11}(s)P_\e(s,1)-\l_{11}(s)P_\e(s,2)=0,\q~\ae~s\in[t,1],\q~\\
\ns\ds P_\e(1,1)=1,
\ea\right. \mbox{for } i=1,$$
and
$$\left\{\ba{ll}
\ds \dot P_\e(s,2)-\frac{1}{\e}P_\e(s,2)^2-\l_{22}(s)P_\e(s,1)+\l_{22}(s)P_\e(s,2)=0,\q \ae~s\in[t,1],\q~\\
\ns\ds P_\e(1,2)=1,
\ea\right. \mbox{for } i=2.$$
Solving the above equations yields
$$P_\e(s,1)=P_\e(s,2)=\frac{\e}{\e+1-s},\q~s\in[0,1].$$
Or
$$P_\e(s,i)=\frac{\e}{\e+1-s},\q~(s,i)\in[0,1]\ts\cS.$$
Noting that the state space of $\a(s)$ is $\cS=\{1,2\}$, we let
\bel{Theta}\ba{ll}
\ds \Th_\e(s,\a(s))\deq-[\hat R_\e(s,\a(s))+\e I_m]^{-1}\hat S_\e(s,\a(s))\\
\ns\ds\qq\qq\q=-\frac{P_\e(s,\a(s))}{\e}=-\frac{1}{\e+1-s},\q~s\in[0,1].
\ea\ee
Then, the corresponding BSDE \rf{eta-zeta-xi} reads
$$\left\{\2n\ba{ll}
\ds d\eta_\e(s)=-\Big\{\big[\Th_\e(s,\a(s))-\a(s)\big]\eta_\e(s)+\sqrt{2\a(s)}\z_\e(s)+P_\e(s,\a(s))b(s)\Big\}ds\\
\ns\ns\ds \qq\qq+\z_\e(s) dW(s)+\sum_{k=1}^2\xi^\e_k(s)d\wt{N}_k(s),\q~s\in[0,1],\\
\ns\ns\ds \eta_\e(1)=0.\ea\right.$$
Let $f(s)=\frac{1}{\sqrt{1-s}}$. Using the variation of constants formula for BSDEs, and noting that $W(\cd)$ and $\wt{N}_k(\cd)$ are $(\dbF, \dbP)$-martingales, we obtain
$$\ba{ll}
\ds \eta_\e(s)=\frac{\e}{\e+1-s}\cd\exp\left\{2\int_0^s\a(r)dr-\int_0^s\sqrt{2\a(r)}dW(r)\right\}\\
\ns\ds\qq\q\ \cd\dbE\left[\int_s^1b(r)
\cd\exp\left\{\int_0^r\sqrt{2\a(\bar{r})}dW(\bar{r})-2\int_0^r\a(\bar{r})d\bar{r}\right\}dr\bigg|\cF_s\right]\\
\ns\ds\qq\ =\frac{\e}{\e+1-s}\cd\exp\left\{2\int_0^s\a(r)dr-\int_0^s\sqrt{2\a(r)}dW(r)\right\}\\
\ns\ds\qq\q\ \cd\int_s^1f(r)
\cd\dbE\left[\exp\left\{2\int_0^r\sqrt{2\a(\bar{r})}dW(\bar{r})
-4\int_0^r\a(\bar{r})d\bar{r}\right\}\bigg|\cF_s\right]dr\\
\ns\ds\qq\ =\frac{\e}{\e+1-s}\cd\exp\left\{\int_0^s\sqrt{2\a(r)}dW(r)-2\int_0^s\a(r)dr\right\}
\cd\int_s^1f(r)dr,\q~s\in[0,1].
\ea$$
It should be point out that, in the above equality, we use the Fibini's Theorem and the martingale property, i.e.,
$$\ba{ll}
\ds \dbE\left[\exp\left\{2\int_0^r\sqrt{2\a(\bar{r})}dW(\bar{r})
-4\int_0^r\a(\bar{r})d\bar{r}\right\}\bigg|\cF_s\right]\\
\ns\ds =\exp\left\{2\int_0^s\sqrt{2\a(\bar{r})}dW(\bar{r})
-4\int_0^s\a(\bar{r})d\bar{r}\right\},\q~0\les s\les r\les1.
\ea$$
Now, let
\bel{Eta}\ba{ll}
\ds v_\e(s)\deq-[\hat R_\e(s,\a(s))+\e I_m]^{-1}\hat\rho_\e(s,\a(s))=-\frac{\eta_\e(s)}{\e}\\
\ns\ds\qq\ =-\frac{1}{\e+1-s}\cd\exp\left\{\int_0^s\sqrt{2\a(r)}dW(r)-2\int_0^s\a(r)dr\right\}
\cd\int_s^1f(r)dr,\q~s\in[0,1].
\ea\ee
Then, the corresponding closed-loop system \rf{eclosed-loop-state} can be written as
$$\left\{\2n\ba{ll}
 \ds dX_\e(s)=\Big\{\big[\Th_\e(s,\a(s))-\a(s)\big]X_\e(s)+v_\e(s)+b(s)\Big\}ds
 +\sqrt{2\a(s)}X_\e(s)dW(s),\q s\in[0,1], \\
 \ns\ns\ds X_\e(0)=x,\ea\right.$$
By the variation of constants formula for SDEs, we get
$$\ba{ll}
\ds X_\e(s)=(\e+1-s)\cd\exp\left\{\int_0^s\sqrt{2\a(r)}dW(r)-2\int_0^s\a(r)dr\right\}\\
\ns\ds\qq\qq \cd\int_0^s\left[\frac{1}{\e+1-r}\cd
\exp\left\{-\int_0^r\sqrt{2\a(\bar{r})}dW(\bar{r})+2\int_0^r\a(\bar{r})d\bar{r}\right\}
\cd\big(v_\e(r)+b(r,\a(r))\big)\right]dr\\
\ns\ds\qq\q~ +x\cd\frac{\e+1-s}{\e+1}\cd\exp\left\{\int_0^s\sqrt{2\a(r)}dW(r)-2\int_0^s\a(r)dr\right\},\q~s\in[0,1].
\ea$$
In light of Theorem \ref{10.23.16.22}, in order to prove the open-loop solvability at $(0,x,i)$, it suffices to show the family $\{u_\e(\cd)\}_{\e>0}$ defined by
\bel{11.4.21.34}\ba{ll}
\ds u_\e(s)\deq\Th_\e(s,\a(s))X_\e(s)+v_\e(s)\\
\ns\ds\qq~ =-\exp\left\{\int_0^s\sqrt{2\a(r)}dW(r)-2\int_0^s\a(r)dr\right\}\\
\ns\ds\qq\qq~ \cd\int_0^s\left[\frac{1}{\e+1-r}\cd
\exp\left\{-\int_0^r\sqrt{2\a(\bar{r})}dW(\bar{r})+2\int_0^r\a(\bar{r})d\bar{r}\right\}
\cd\big(v_\e(r)+b(r,\a(r))\big)\right]dr\\
\ns\ds\qq\q~ -\frac{x}{\e+1}\cd\exp\left\{\int_0^s\sqrt{2\a(r)}dW(r)-2\int_0^s\a(r)dr\right\}
+v_\e(s),\q~s\in[0,1],
\ea\ee
is bounded in $L^2_{\dbF}(0,1;\dbR)$. For this, let us first simplify \rf{11.4.21.34}. On the one hand, by Fubini's theorem,
$$\ba{ll}
\ds \int_0^s\left[\frac{1}{\e+1-r}\cd
\exp\left\{-\int_0^r\sqrt{2\a(\bar{r})}dW(\bar{r})+2\int_0^r\a(\bar{r})d\bar{r}\right\}
\cd v_\e(r)\right]dr\\
\ns\ds =-\int_0^s\frac{1}{(\e+1-r)^2}\int_r^1f(\bar{r})d\bar{r}dr\\
\ns\ds =-\int_0^sf(\bar r)\int_0^{\bar r}\frac{1}{(\e+1-r)^2}drd\bar{r}
-\int_s^1f(\bar r)\int_0^s\frac{1}{(\e+1-r)^2}drd\bar r\\
\ns\ds =-\int_0^s\frac{1}{\e+1-r}\cd f(\bar r)d\bar{r}+\frac{1}{\e+1}\int_0^1f(\bar r)d\bar r
-\frac{1}{\e+1-r}\int_s^1f(\bar r)d\bar r.
\ea$$
Similarly, on the other hand,
$$\ba{ll}
\ds \int_0^s\left[\frac{1}{\e+1-r}\cd
\exp\left\{-\int_0^r\sqrt{2\a(\bar{r})}dW(\bar{r})+2\int_0^r\a(\bar{r})d\bar{r}\right\}
\cd b_\e(r,\a(r))\right]dr=\int_0^s\frac{1}{\e+1-r}f(r)dr.
\ea$$
Consequently, we get
\bel{11.5.20.56}\ba{ll}
\ds u_\e(s)=-\left(\frac{x}{\e+1}+\frac{1}{\e+1}\int_0^1f(r)dr\right)
\cd\exp\left\{\int_0^s\sqrt{2\a(r)}dW(r)-2\int_0^s\a(r)dr\right\}\\
\ns\ds\qq\ =-\frac{x+2}{\e+1}\cd\exp\left\{\int_0^s\sqrt{2\a(r)}dW(r)-2\int_0^s\a(r)dr\right\}.
\ea\ee
A short calculation gives
$$\dbE\int_0^1|u_\e(s)|^2ds=\left(\frac{x+2}{\e+1}\right)^2\les (x+2)^2,\q~\forall \e>0.$$
Therefore, $\{u_\e(\cd)\}_{\e>0}$ is bounded  in $L^2_{\dbF}(0,1;\dbR)$. Now, let $\e\ra0$ in \rf{11.5.20.56}, we get an open-loop optimal control:
$$u^*(s)=-(x+2)\cd\exp\left\{\int_0^s\sqrt{2\a(r)}dW(r)-2\int_0^s\a(r)dr\right\},\q~s\in[0,1].$$
From the above discussion, similar to the state process $X(\cd)$ of \rf{11.5.21.27}, the open-loop optimal control $u^*(\cd)$ also depends on the regime switching term $\a(\cd)$. That is to say, as the value of the switching $\a(\cd)$ varies, the open-loop optimal control $u^*(\cd)$ will be changed too.

\ms

Finally, we let $\e\ra0$ in \rf{Theta} and \rf{Eta} to get a weak closed-loop optimal strategy
$(\Th^*(\cd,\cd),v^*(\cd))$:
$$\ba{ll}
\ds \Th^*(s,\a(s))=\lim_{\e\ra0}\Th_\e(s,\a(s))=-\frac{1}{1-s},\qq s\in[0,1),\\
\ns\ds v^*(s)=\lim_{\e\ra0}v_\e(s)=
-\frac{1}{1-s}\cd\exp\left\{\int_0^s\sqrt{2\a(r)}dW(r)-2\int_0^s\a(r)dr\right\}
\cd\int_s^1f(r)dr\\
\ns\ds\qq\qq\qq\q~=-\frac{2}{\sqrt{1-s}}\cd\exp\left\{\int_0^s\sqrt{2\a(r)}dW(r)-2\int_0^s\a(r)dr\right\},
\qq s\in[0,1).
\ea$$
We put out that neither $\Th^*(\cd,\cd)$ and $v^*(\cd)$ is square-integrable on $[0,1)$. Indeed, one has
$$\ba{ll}
\ds \dbE\int_0^1|\Th^*(s,\a(s))|^2ds=\int_0^1\frac{1}{(1-s)^2}ds=\i,\\
\ns\ns\ds \dbE\int_0^1|v^*(s)|^2ds
=\dbE\int_0^1\frac{4}{1-s}\cd\exp\left\{2\int_0^s\sqrt{2\a(r)}dW(r)-4\int_0^s\a(r)dr\right\}ds\\
\ns\ns\ds\qq\qq\qq\q=\dbE\int_0^1\frac{4}{1-s}ds=\i.
\ea$$
\end{example}

\section{Conclusions}

In this paper, we mainly study the open-loop and weak closed-loop solvabilities for a class of stochastic LQ optimal control problems of Markovian regime switching system. The main result is that these two solvabilities are equivalent.
First, using the perturbation approach, we provide an alternative characterization of the open-loop solvability. Then we investigate the weak closed-loop solvability of the LQ problem of Markovian regime switching system, and establish the equivalent relationship between open-loop and weak closed-loop solvabilities. Finally, we present an example to illustrate the procedure for finding weak closed-loop optimal strategies in the circumstance of Markovian regime switching system.

\section*{Competing interests}

Conflict of Interest: The authors declare that they have no conflict of interest.

\section*{Acknowledgements}

The authors would like to thank the editors and the anonymous referees for many helpful comments
and valuable suggestions on this paper.

\end{document}